\documentclass[11pt,a4paper]{amsart}
\usepackage[pdftex,colorlinks=true,
                      pdfstartview=FitV,
                      linkcolor=blue,
                      citecolor=blue,
                      urlcolor=blue,
          ]{hyperref}
\usepackage{amsfonts}
\usepackage{amssymb}
\usepackage{diagrams}
\newarrow{RLCorresponds}{harpoon}---{harpoon}
\newtheorem{thm}{Theorem}[section]
\newtheorem{lemma}[thm]{Lemma}
\newtheorem{cor}[thm]{Corollary}
\newtheorem{exs}[thm]{Examples}
\newtheorem{prp}[thm]{Proposition}
\newtheorem{df}[thm]{Definition}
\newtheorem{rem}[thm]{Remark}

\newtheorem{quest}{Question}[section]

\newcommand{\tr}{{\rm tr}}
\newcommand{\Aut}{{\rm Aut}}
\newcommand{\Dim}{{\rm Dim}}

\newcommand{\Hom}{\mbox{{\rm Hom}}}
\newcommand{\End}{\mbox{{\rm End}}}
\newcommand{\Mor}{\mbox{{\rm Mor}}}
\newcommand{\D}{{\mathcal D}}
\newcommand{\E}{{\mathcal E}}
\newcommand{\F}{{\mathcal F}}

\newcommand{\calS}{\mathcal{S}}
\newcommand{\C}{\mathfrak{C}}
\newcommand{\Pp}{\mathcal{P}}

\newcommand{\ts}{\mathfrak{Ts}}
\newcommand{\fld}{{k}}

\author{Peter Fleischmann and Chris Woodcock}

\title [$p$-Group Galois Extensions]{Free group actions on varieties and the category of modular Galois Extensions\\ for
finite $p$-groups}

\begin{document}

\begin{abstract}
Let $K$ be an algebraically closed field and $\mathbb{A}^n\cong K^n$
affine $n$-space. It is known that a finite group $\frak{G}$ can only act freely on $\mathbb{A}^n$
if $K$ has characteristic $p>0$ and $\frak{G}$ is a $p$-group. In that case
the group action is ``non-linear" and the ring of regular functions
$K[\mathbb{A}^n]$ must be a \emph{trace-surjective} $K-\frak{G}$-algebra.\\
Now let $\fld$ be an arbitrary field of characteristic $p>0$ and let $G$ be a finite $p$-group.
In this paper we study the category $\ts$ of all finitely generated trace-surjective $\fld-G$ algebras.
It has been shown in \cite{nonlin} that the objects in $\ts$ are precisely those finitely
generated $\fld-G$ algebras $A$ such that $A^G\le A$ is a Galois-extension in the sense
of \cite{chr}. Although $\ts$ is not an abelian category it has
``$s$-projective objects", which are analogues of projective modules, and it has ($s$-projective)
categorical generators, which we will describe explicitly.
We will show that $s$-projective objects and their rings of invariants are retracts of
polynomial rings and therefore regular UFDs. The category $\ts$ also has
``weakly initial objects", which are closely related to the essential dimension of $G$ over $\fld$.
Our results yield a geometric structure
theorem for free actions of finite $p$-groups on affine $\fld$-varieties.
There are also close connections to open questions on retracts of polynomial rings,
to embedding problems in standard modular Galois-theory of $p$-groups and,
potentially, to a new constructive approach
to homogeneous invariant theory.
\end{abstract}

\maketitle

{
\renewcommand{\thefootnote}{}
\addtocounter{footnote}{-1}
}

\setcounter{section}{-1}
\small
\section{Introduction} \label{Intro}

Let $\fld$ be a field, $\frak{G}$ a finite group and $X$ a $\fld$-variety.
The following beautiful argument appears in Serre's paper
``How to use finite fields for problems concerning infinite fields"
(\cite{serre_how_to}).
Unable to express it any better we quote almost verbatim:\\
`` Suppose that $\frak{G}$ acts freely on $X$. There is a
Cartan-Leray spectral sequence (... of \'etale cohomology...)
$H^i(\frak{G},H^j(X,C)) \Rightarrow H^{i+j}(\frak{G},H^j(X/\frak{G},C)),$
where $C$ is any finite abelian group. If $X$ is the affine $n$-space
$\mathbb{A}^n$ and $|C|$ is prime to ${\rm char}(\fld)$,
then $H^j(X,C)=0$ for $j>0$ and $H^0(X,C)=C$. In that case
the spectral sequence degenerates and gives $H^i(\frak{G},C)=H^i(X/\frak{G},C)$
for every $i$, i.e. $X/\frak{G}$ has the same cohomology as the classifying space
of $\frak{G}$. Take now $C=\mathbb{Z}/\ell\mathbb{Z}$ and suppose that $\ell$ divides $|\frak{G}|$.
It is well known that $H^j(\frak{G},C)$ is non-zero for infinitely many $j$'s, and
that $H^j(X/\frak{G},C)$ is zero for $j>2\cdot {\rm dim} X$:\ contradiction!"
\\[1mm]
This establishes the following
\begin{thm}\label{free_gp_act_on A_n}
The only finite groups which can act freely on $\mathbb{A}^n$ are the $p$-groups
with $p={\rm char}(\fld)$.
\end{thm}
Serre then poses the \emph{Exercise}:\ ``Let $\frak{G}$ be a finite $p$-group with $p={\rm char}(\fld)$.
Show that there exists a free action on $\mathbb{A}^n$, provided that $n$ is large enough."\\
Parts of the current article can be viewed as solving a ``generic version" of this exercise.
Using results from \cite{nonlin} we obtain the following:

\begin{thm}\label{A_times_B}
Let $\fld=\overline{\fld}$ be an algebraically closed field of characteristic $p>0$ and $G$ be a
finite group of order $p^n$. Then the group $G$ acts freely on the affine space ${\bf A}\cong\fld^{|G|-1}$
in such a way that  the following hold:
\begin{enumerate}
\item The quotient space ${\bf A}/G$ is isomorphic to affine space $k^{|G|-1}$.
\item There is a (non-linear) decomposition ${\bf A}={\bf B}\times{\bf C}$
such that $G$ acts freely on ${\bf B}\cong k^n$ and trivially on ${\bf C}\cong k^{|G|-n-1}$.
\item The quotient space ${\bf B}/G$ is isomorphic to affine space $k^n$.
\end{enumerate}
\end{thm}

Moreover we will show that the varieties ${\bf A}$ and ${\bf B}$ are \emph{cogenerators}
in the category of affine varieties with free $G$-action.
Combining this with a structure theorem in \cite{nonlin} on modular Galois-extensions of finite
$p$-groups, we obtain the following \emph{geometric structure theorem}:

\begin{thm}\label{geom_struct_thm}
Let $\fld=\overline{\fld}$ be an algebraically closed field of characteristic $p>0$ and $G$ be a
finite group of order $p^n$ and let $X$ be an  arbitrary  affine variety.
\begin{enumerate}
\item There is an affine variety $Y$ with free $G$-action such
that $Y/G\cong X$.
\item Every such $Y$ is a fibre product of the form $Y\cong X\times_{{\bf B}/G}{\bf B}$.
\item For every such $Y$ there is a $G$-equivariant embedding
$Y\hookrightarrow{\bf B}^N$ for some $N\in\mathbb{N}$ (which is the ``cogenerator
property" of ${\bf B}$).
\end{enumerate}
\end{thm}
It turns out that free actions of $p$-groups on affine varieties in characteristic $p>0$ are dualizations
of group actions on affine $k$-algebras which are Galois ring extensions over the
ring of invariants, in the sense of Auslander-Goldmann \cite{AG} or
Chase-Harrison-Rosenberg \cite{chr}. In \cite{nonlin} we showed that for a $p$-group $G$ acting on
a $k$-algebra $A$ in characteristic $p$, the extension $A\ge A^G$ is Galois if and only if the algebra $A$ is \emph{trace-surjective}
in the sense of Definition \ref{ts-alg_df}. We then went on to develop a structure theory for
such algebras and their rings of invariants. Using the results obtained there,
we will prove Theorems \ref{A_times_B} and \ref{geom_struct_thm} by studying
the category of modular Galois extensions of finitely generated $\fld$-algebras,
where the Galois group is a fixed finite $p$-group.
\\[1mm]
Let $\frak{G}$ be an arbitrary finite group, $\fld$ a field and $A$ a commutative
$\fld$-algebra on which $\frak{G}$ acts by $\fld$-algebra automorphisms;
then we call $A$ a $\fld-\frak{G}$ algebra.
Let $A^\frak{G}:=\{a\in A\ |\ ag=a\ \forall g\in \frak{G}\}$ be the \emph{ring of invariants} and
let $\tr:=\tr_\frak{G}:\ A\to A^\frak{G},\ a\mapsto \sum_{g\in \frak{G}} ag$
be the \emph{transfer map} or
\emph{trace map}. This is obviously a homomorphism of $A^\frak{G}$-modules, but not of $\fld $-algebras.
As a consequence the image ${\rm tr}(A)\unlhd A^\frak{G}$ is an ideal in $A^\frak{G}$.

\begin{df}\label{ts-alg_df}
A $\fld-\frak{G}$ algebra $A$ such that $\tr(A)=A^\frak{G}$ will be called a
{\bf trace-surjective} $\fld-\frak{G}$-algebra.
With $\ts:=\ts_\frak{G}$ we denote the category of all finitely generated trace-surjective
$\fld-\frak{G}$-algebras, with morphisms being $\frak{G}$-equivariant homomorphisms of $\fld$-algebras.
For $A,B\in\ts$ the set of morphisms $\phi:\ A\to B$ will be denoted by $\ts(A,B)$.
\end{df}

The category $\ts$ contains \emph{weakly initial} objects $\frak{W}\in\ts$ satisfying
$\ts(\frak{W},A)\ne\emptyset$ for any $A\in\ts$. Every algebra $A\in\ts$ turns out to be
an extension by invariants of a quotient of $\frak{W}$ of the form
$A^G\otimes_{X^G}X$, where $X\cong \frak{W}/I$ for some $G$-stable ideal $I\unlhd \frak{W}$
(see Lemma \ref{char_universal}). This is why we call the weakly initial objects in $\ts$
``universal" trace surjective algebras.

The category $\ts$ is not abelian. However, it has finite coproducts given by tensor
products of $\fld$-algebras. With the help of these
one can define analogues of projective modules, which we call
``$s$-projective objects", because projectivity is defined using surjective maps rather
than epimorphisms. There are also analogues of generators in module categories and
we will give explicit descriptions of $s$-projective generators.
These arise in (homogeneous) modular invariant theory as dehomogenized symmetric algebras of suitable
linear representations, such as the regular representation.
Let $S\hookrightarrow T$ be an extension of $\fld$-algebras, then $S$ is a \emph{retract}
of $T$ if $T=S\oplus I$ with \emph{ideal} $I\unlhd T$. We will show
that $s$-projective objects and their rings of invariants are retracts of
polynomial rings and therefore regular Unique Factorization Domains (UFDs) (see \cite{costa} Proposition 1.8).

From now on let $\fld$ be an arbitrary field of characteristic $p>0$ and $G$ a finite $p$-group.
We will adopt the following definitions and notations, often used in affine algebraic geometry:
\begin{df}\label{stably_def} Let $R$ be a $\fld$-algebra and $n\in\mathbb{N}$.
\begin{enumerate}
\item With $R^{[n]}$ we denote the polynomial ring $R[T_1,\cdots,T_n]$ over $R$.
\item Let $\mathbb{P}=\fld[T_1,\cdots,T_m]\cong\fld^{[m]}$ and $G\le \Aut_\fld(\mathbb{P})$.
Then $\mathbb{P}$ is called {\bf triangular} (with respect to the chosen generators $T_1,\cdots,T_m$),
if for every $g\in G$ and $i=1,\cdots,m$ there is $f_{g,i}(T_1,\cdots,T_{i-1})\in \fld[T_1,\cdots,T_{i-1}]$ such that $(T_i)g=T_i+f_{g,i}(T_1,\cdots,T_{i-1}).$
\item Let $m\in\mathbb{N}$, then a $\fld$-algebra $R$ is called {\bf ($m$-) stably polynomial} if
$T:=R\otimes_\fld \fld^{[m]}\cong R^{[m]}\cong \fld^{[N]}$ for some $N\in\mathbb{N}$. Assume
moreover that $R$ is a $\fld-G$ algebra and $T$ extends the $G$-action on $R$ trivially, i.e.
$T\cong R\otimes_\fld F$ with $F=F^G\cong\fld^{[m]}$. If $T$ is triangular, then we call
$R$ {\bf ($m$-) stably triangular}.
\end{enumerate}
\end{df}

In order to describe the results of this paper in more detail, we need to refer to some definitions and results obtained in \cite{nonlin}:
\\[1mm]
Let $G$ be a finite group of order $p^n$ with regular representation $V_{reg}\cong kG$
and let $D_\fld$ be the dehomogenization of the symmetric algebra ${\rm Sym}(V_{reg}^*)$,
as defined in \cite{nonlin} (see also Section \ref{sec_Free affine actions}
shortly after Theorem \ref{first_main}).
It is known that a graded algebra and its dehomogenizations share many interesting properties (see e.g.
\cite{BH} pg. 38 and the exercises 1.5.26, 2.2.34, 2.2.35 loc. cit.)
Clearly the algebra $D_\fld\in\ts$ is a polynomial ring of Krull-dimension $|G|-1$ with
triangular $G$-action.

The following Theorem was one of the main results of \cite{nonlin}:
\begin{thm}[\cite{nonlin} Theorems 1.1-1.3] \label{arb_p_grp_sec0}
There exists  a trace-surjective triangular $G$-subalgebra $U:=U_G\le D_\fld$,
such that $U\cong \fld^{[n]}$ is a retract of $D_\fld$, i.e. $D_\fld=U\oplus I$ with a $G$-stable \emph{ideal} $I\unlhd D_\fld$.
Moreover: $U^G\cong \fld^{[n]}$ and $D_\fld^G\cong \fld^{[|G|-1]}$.
\end{thm}

\noindent
For any $\fld-G$-algebra $A\in\ts$ and $\ell\in\mathbb{N}$ we define
$A^{\otimes\ell}:=\coprod_{i=1}^\ell A:= A\otimes_\fld \cdots\otimes_\fld A$
with $\ell$ copies of $A$ involved. The following are main results of the present
paper:

\begin{thm}\label{cat_gal_main1}
Let $\Gamma\cong\fld^{[d]}\in\ts$ with triangular $G$-action, e.g. $\Gamma\in\{D_\fld,U\}$.
Then \begin{enumerate}
\item $\Gamma$ is an $s$-projective generator in the category $\ts$.
\item For any $A\in\ts$ there is a $G$ equivariant isomorphism $A\otimes_\fld\Gamma\cong A\otimes_\fld\fld[T_1,\cdots,T_d]\cong A^{[d]}$, which is the identity on $A$, with $$\fld[T_1,\cdots,T_d]\le (A\otimes_\fld\Gamma)^G\cong (A^G)^{[d]}.$$
\item $\Gamma^{\otimes\ell}\cong\Gamma\otimes_\fld \fld[s_1,\cdots,s_N]$
with $\fld^{[N]}\cong\fld[s_1,\cdots,s_N]\le (\Gamma^{\otimes\ell})^G$.
\item For every $\ell$ the ring of invariants $(\Gamma^{\otimes\ell})^G$ is stably polynomial.
\item If $\Gamma\in\{D_\fld, U\}$, then $(\Gamma^{\otimes\ell})^G$ is a polynomial ring.
\end{enumerate}
\end{thm}
\begin{proof}
(1),(2) and (3):\ It follows from Proposition \ref{triangular implies erasable}
that $\Gamma$ is ``erasable" (see Definition \ref{triang_dissolv}), which by
Theorem \ref{erasbility_theorem} implies that $\Gamma$ is an $s$-projective generator in $\ts$.\\
(4):\ This follows from Theorem \ref{erasable_is_poly_bd_stably_inv}.\\
(5):\ This follows from Corollary \ref{erasable_is_poly_bd_stably_inv_cor1}.
\end{proof}

\begin{thm}\label{cat_gal_main2}
Let $\Pp\in\ts$ be s-projective\footnote{see Definition \ref{universal_df}}, then both, $\Pp$ and $\Pp^G$ are retracts of polynomial
rings over $\fld$.
\end{thm}
\begin{proof}
See Theorem \ref{cat_gal_main2_pf}.
\end{proof}

It follows from \cite{costa} Proposition 1.8 that retracts of
a unique factorization domain (UFD) are UFDs as well and
from \cite{costa} Corollary 1.11 that retracts of regular rings are regular.
Hence
\begin{cor}\label{cat_gal_main2_cor}
Let $\Pp\in\ts$ be s-projective, then both, $\Pp$ and $\Pp^G$ are regular UFDs.
\end{cor}

Setting ${\bf A}:={\rm max-spec}(D_\fld)$ and
${\bf B}:={\rm max-spec}(U)$ with ${\bf A}/G\cong {\rm max-spec}(D_\fld^G)$
and ${\bf B}/G\cong {\rm max-spec}(U^G)$
it is clear now how to obtain Theorems \ref{A_times_B} and \ref{geom_struct_thm} (3)
from Theorems \ref{arb_p_grp_sec0} and \ref{cat_gal_main1}. The statements
in \ref{geom_struct_thm} (1) and (2) follow from
\begin{thm}[\cite{nonlin} Theorem 1.2]\label{strct_thm_intro}
Every algebra $A\in \ts$ with given ring of invariants $A^G=R$ is of the form
$$A\cong R[Y_1,\cdots,Y_n]/(\sigma_1(\underline Y)-r_1,\cdots,\sigma_n(\underline Y)-r_n)$$
with suitable $r_1,\cdots,r_n\in R$, and $G$-action derived from the action on $U$.
\end{thm}

Let $V$ be a finite dimensional $\fld$-vector space, $\frak{G}\le{\rm GL}(V)$ a finite group
and $S(V^*):={\rm Sym}(V^*)$ the symmetric algebra over the dual space $V^*$ with induced linear $\frak{G}$-action.
One of the main objectives of (homogeneous)
invariant theory is the study of the structure of the ring of invariants $S(V^*)^\frak{G}$.
By a result of Serre (\cite{bou}) these rings are regular (and then polynomial, as they are graded rings)
only if the group $\frak{G}$ is generated by pseudo-reflections. If ${\rm char}(\fld)$ does not divide $|\frak{G}|$,
the converse also holds by the well-known theorem of Chevalley-Shephard-Todd and Serre(see e.g. \cite{B} or
\cite{LSb}). If $\frak{G}=G$ is a $p$-group in characteristic $p>0$, all pseudo-reflections are
transvections of order $p$, so if $G$ is not generated by elements of order $p$ the
ring $S(V^*)^G$ can never be regular. In this case $S(V^*)^G$ can have a very complicated
structure and, in most cases, will not even be Cohen-Macaulay.
If $A\in\ts$, then obviously $S(V^*)\otimes_\fld A\in\ts$.
Using the universal property of polynomial rings one can show that for every
$\fld$-$G$-algebra $S\cong\fld^{[d]}$ with triangular $G$-action,
the $\fld$-$G$ algebra $S\otimes_k\Pp$ is s-projective in $\ts$, whenever $\Pp$ is.
In particular $S\otimes_\fld \Pp$ and $(S\otimes_\fld \Pp)^G$ are retracts of polynomial rings
and therefore regular UFDs.
\\
In \cite{costa}, the question was asked whether retracts of polynomial rings are
again polynomial rings. Despite some positive answers in low-dimensional special cases
(see \cite{shpilrain}) this question was unanswered for several decades.
Recently S. Gupta (\cite{gupta:2013}) found a counterexample to the ``cancellation problem" in characteristic $p>0$, which also
implies a negative answer in general to Costa's question. Gupta's example yields a
non-polynomial retract $R$ of a polynomial ring, which however is still stably polynomial.
Using Theorem \ref{strct_thm_intro} one can easily construct $A\in\ts$ with $A^G\cong R$, such
that $A$ is s-projective. So there are s-projective objects in $\ts$ with non-polynomial
invariant rings. If all retracts of polynomial rings were \emph{stably} polynomial, then this
would be true for arbitrary s-projective objects in $\ts$ and their invariant rings.
This is our main reason for the following
\begin{quest}\label{conjecture}
Are $\Pp$ and $\Pp^G$ stably polynomial rings for  every s-projective $\Pp\in\ts$?
\end{quest}

For $\Pp=D_\fld$ or $U$ this is already contained in Theorem \ref{arb_p_grp_sec0} and for
$\Pp=S\otimes D_\fld$ or $\Pp=S\otimes U$ with triangular $\fld-G$ algebra
$S\cong\fld^{[d]}$ it follows from \ref{erasable_inv_is_poly_intersetc_poly}.
From this one can derive a result that includes ``graded modular rings of invariants",
for which we don't know any other reference in the literature:

\begin{thm}\label{SVG_is_poly_inters_poly_intro}
Let $S\cong \fld^{[d]}$ be a polynomial ring with triangular $G$-action (e.g. $S=S(V^*)$).
Then the ring of invariants $S^G$ is the intersection of two polynomial subrings inside an
$s$-projective polynomial $\fld-G$-algebra $\fld^{[N]}\in\ts$ of Krull-dimension $N=d+n$ with
$n:=\log_p{|G|}$. If moreover $S\in\ts$, then $S^G\otimes_\fld\fld^{[n]}\cong\fld^{[n+d]}$,
i.e. $S^G$ is $n$-stably polynomial.
\end{thm}
\begin{proof} See Theorem \ref{SVG_is_poly_inters_poly_pf}.
The proof will show that the intersection $S^G$ can be obtained by a procedure
of ``elimination of variables".
\end{proof}

A special role in the category $\ts$ is played by  ``minimal universal" algebras, which are investigated
in Sections \ref{sec_Minimal universal objects} and \ref{sec_Basic Algebras}.
They turn out to be integral domains of the same Krull-dimension $d_\fld(G)$,
an invariant depending only on the group $G$ and the field $\fld$ and an upper bound
for the ``essential dimension" $e_\fld(G)$ as defined by Buhler and Reichstein (\cite{Buhler:Reichstein}).
The following is one of the main results of these sections: (See Section \ref{sec_Basic Algebras} and Theorem \ref{ess_dim_squeeze} for details and precise definitions).

\begin{thm}\label{ess_dim_squeeze_intro}
Let ${\rm char}(\fld)=p>0$ and $G$ be a group of order $p^n$.
The minimal universal objects $\frak{U}\in\ts$ are integral domains of Krull dimension
$d_\fld(G)$, satisfying $e_\fld(G)\le d_\fld(G)\le n$.
Moreover, ``essential $G$-fields" of transcendence degree $e_\fld(G)$ appear among
the ``embedded residue class fields" $\fld(\wp)\hookrightarrow {\rm Quot}(\frak{U})$ of $\frak{U}$ with
respect to suitable $G$-stable prime ideals $\wp\unlhd\frak{U}$.
\end{thm}

The rest of the paper is organized as follows: In \emph{Section one} we describe the connection between free actions of a finite group
on affine varieties and Galois extensions of rings. In particular for normal varieties
we formulate a freeness-criterion in terms of the Dedekind different (Corollary \ref{intro_cor_2}).
We will also introduce some basic notation and describe results from previous work, which will
be needed in the sequel.
From there on, $\fld$ will always be a field of characteristic $p>0$ and $G$ will be a
finite $p$-group. In \emph{Section two} we introduce and analyze the {\sl universal},
{\sl projective} and {\sl generating} objects in $\ts$. We also introduce the notion
of {\sl erasable} algebras, which will lead to proofs of the main results,
Theorems \ref{cat_gal_main1}, \ref{cat_gal_main2} and \ref{SVG_is_poly_inters_poly_intro}.
In \emph{Section three} we turn our attention to \emph{basic} algebras, which
we define as \emph{minimal} universal algebras in $\ts$. We classify all basic
algebras which are also normal rings, in the case where $G$ is elementary-abelian of
rank $n$ and $\dim_{\mathbb{F}_p}(\fld)\ge n$. They all turn out to be univariate polynomial algebras with explicitly described non-linear $G$-action. Moreover, in this case the basic normal algebras in
$\ts$ coincide with the minimal normal generators and minimal
normal $s$-projective objects (see Theorem \ref{el_abel_ext_fld_d_is_1}).
The connection between basic algebras and the essential dimension of $G$ over $\fld$ and the proof of Theorem \ref{ess_dim_squeeze_intro} is the topic of \emph{Section four}.
The brief final \emph{Section five} contains an open question and a conjecture.

\section{Free affine actions and Galois-extensions} \label{sec_Free affine actions}

Free group actions on affine varieties are closely related to Galois ring extensions,
as we will now demonstrate.
\\[1mm]
First let $\frak{G}$ be an arbitrary finite group and $A$ a finitely generated commutative
$\fld-\frak{G}$ algebra. We want to keep flexibility between left and right group actions;
therefore in whatever way
the ``natural side" of the action is chosen, we will use the rule
$g f:=f\cdot g^{-1}$ to switch freely between left and right actions when convenient.
\footnote{If $A$ is an algebra of $\fld$-valued functions on a $\frak{G}$-set $X$, (e.g. $A=\fld[X]$,
the algebra of regular functions on a variety $X$ with $\frak{G}\le {\rm Aut}(X)$)
there is a natural right action of $\frak{G}$ on $A$ given by composition $f\circ g$ for
$g\in \frak{G}$. In other situations we might have a given linear left
$\frak{G}$-action defined on a $k$-vector space $\Omega:=\sum_{i=1}^m\fld\omega_i$.
This extends to a natural \emph{left} action on the symmetric algebra
${\rm Sym}_\fld(\Omega)=\fld[\omega_1,\cdots,\omega_m]$ by
$g(\omega_1^{e_1}\cdots\omega_m^{e_m}):=(g\omega_1)^{e_1}\cdots(g\omega_m)^{e_m}.$}

Set $B:=A^\frak{G}$ and define $\Delta:=\frak{G}\star A=A\star \frak{G}:=\oplus_{g\in \frak{G}}d_gA$ to be the crossed product of
$\frak{G}$ and $A$ with $d_gd_h=d_{gh}$ and $d_ga=g(a)\cdot d_g=(a)g^{-1}\cdot d_g$ for
$g\in \frak{G}$ and $a\in A$. Let $_BA$ denote $A$ as left $B$-module, then there is a
homomorphism of rings
$$\rho:\ \Delta\to \End(_BA),\ ad_g\mapsto \rho(ad_g)=(a'\mapsto a\cdot g(a')=a\cdot (a')g^{-1}).$$
One calls $B\le A$ a \emph{Galois-extension} with group $\frak{G}$ if $_BA$ is finitely
generated projective and $\rho$ is an isomorphism of rings. This definition goes back
to Auslander and Goldmann \cite{AG} (Appendix, pg.396)
and generalizes the classical notion of Galois field extensions. It also applies
to non-commutative $k-\frak{G}$ algebras, but if $A$ is commutative, this definition of
`Galois-extension' coincides with the one given by Chase-Harrison-Rosenberg in \cite{chr},
where the extension of commutative rings $A^\frak{G}\le A$ is called a Galois-extension
if there are elements $x_1,\cdots,x_n$, $y_1,\cdots,y_n$ in $A$ such that
\begin{equation}\label{chr_def}
\sum_{i=1}^n x_i(y_i)g=\delta_{1,g}:=
\begin{cases}
1& \text{if}\ g=1\\
0&\text{otherwise}.
\end{cases}
\end{equation}

In \cite{chr} the following has been shown:
\begin{thm}(Chase-Harrison-Rosenberg)\cite{chr}\label{chase_harrison_rosenberg}
$A^\frak{G}\le A$ is a Galois extension if and only if
for every $1\ne \sigma\in \frak{G}$ and maximal ideal ${\rm p}$ of $A$ there
is $s:=s({\rm p},\sigma)\in A$ with $s-(s)\sigma\not\in{\rm p}$.
\end{thm}

Now, if $X$ is an affine variety over the algebraically closed field $\fld$,
with $\frak{G}\le {\rm Aut}(X)$ and $A:=\fld[X]$ (the ring of regular functions),
then for every maximal ideal ${\rm m}\unlhd A$, $A/{\rm m}\cong \fld$. Hence if $({\rm m})g={\rm m}$, then
$a-(a)g\in {\rm m}$ for all $a\in A$. Therefore we conclude

\begin{thm}\label{aff act_and_gal thm1}
The finite group $\frak{G}$ acts freely on $X$ if and only if $k[X]^\frak{G}\le k[X]$ is a Galois-extension.
\end{thm}

If $B\le A$ is a Galois-extension, then it follows from equation (\ref{chr_def}),
that ${\rm tr}(A)=A^\frak{G}=B$ (see \cite{chr}, Lemma 1.6), so
$A$ is a \emph{trace-surjective} $\fld-\frak{G}$ algebra.
It also follows from Theorem \ref{chase_harrison_rosenberg},
that for a $p$-group $G$ and $\fld$ of characteristic $p$,
the algebra $A$ is trace-surjective if and only if $A\ge A^G=B$ is a Galois-extension
(see \cite{nonlin} Corollary 4.4.). Using Theorem \ref{free_gp_act_on A_n} we obtain
\begin{cor}\label{ts_gal1}
Let $\fld$ be algebraically closed. Then the finite group $\frak{G}$ acts freely on $X\cong \mathbb{A}^n$ if and only if
$\frak{G}$ is a $p$-group with $p={\rm char}(\fld)$ and $\fld[X]$
is a trace-surjective $\fld-\frak{G}$ algebra.
\end{cor}

Since for $p$-groups in characteristic $p$ the trace-surjective algebras coincide with
Galois-extensions over the invariant ring, we obtain from Theorem
\ref{aff act_and_gal thm1}:

\begin{cor}\label{intro_cor_1}
If $\fld$ is an algebraically closed field of characteristic $p>0$, $X$ an affine $\fld$-variety
and $G$ a finite $p$-group, then
$G$ acts freely on $X$ if and only if $A=\fld[X]\in \ts$.
\end{cor}

Any finite $p$-group $G$ can be realized as a subgroup of some ${\rm SL}_n(\fld)$. The
left multiplication action of $G$ on ${\rm Mat}_n(k)$ induces a homogeneous right regular action on the coordinate ring
$k[M]:=k[{\rm Mat}_n(k)]\cong k[X_{ij}\ |\ 1\le i,j\le n]$ with ${\rm det}:={\rm det}(X_{ij})\in k[M]^G$.
It can be shown that ${\rm det}\in\sqrt{\tr(k[M])}$, in other words
$\tr(f)=({\rm det})^N$ for some $N\in \mathbb{N}$ and some $f\in k[M]$.
It follows that the coordinate ring $k[{\rm GL}_n]=k[M][1/{\rm det}]$ is a
trace-surjective $G$-algebra. Since epimorphic images of trace-surjective algebras are
again trace-surjective (see Theorem \ref{first_main} (iii)),
a similar conclusion holds if ${\rm GL}_n$ is replaced by an arbitrary closed linear algebraic
subgroup $H$ containing $G$ (see \cite{nonlin} Corollary 4.5, where this is proved in a different way).
In particular, if $H={\bf U}$ is a connected unipotent subgroup with
${\bf U}\cong\mathbb{A}^n$, then we obtain the free $G$-action asked for in
Serre's exercise.
\\[2mm]
In the case of a normal affine variety, associated to an affine $\fld-\frak{G}$ algebra which is also a normal
noetherian domain, there is a nice and useful characterization of Galois-extensions
in terms of the Dedekind-different.\footnote{which in the circumstances considered
coincides with E Noether's ``homological different".}
Set $B:=A^\frak{G}$ and $A^\vee:=\Hom_B(_BA,B)$. Then $A^\vee$ is an $A$-module via
$a\cdot\lambda(a')=\lambda(a'a)$ for $a,a'\in A$ and $\lambda\in A^\vee$.
Moreover $A^\vee$ is an $A$-submodule of $\End(_BA)$
and for $\frak{G}^+:=\sum_{g\in \frak{G}}d_g\in \Delta$ we have
$\rho(\frak{G}^+\cdot a)(a')=\tr(aa')=(a\cdot\tr)(a')$, so
$\rho(\frak{G}^+\cdot A)=A\cdot\tr\subseteq A^\vee$.
If in addition $A$ is a normal noetherian domain, then we define
$\D_{A,B}^{-1}:=\{x\in {\rm Quot}(A)\ |\ \tr_\frak{G}(xA)\subseteq B\}$,
the inverse of the Dedekind-different. In this case
the field extension $\mathbb{L}:={\rm Quot}(A)\ge \mathbb{K}:={\rm Quot}(B)=\mathbb{L}^\frak{G}$ is
Galois, so normal and separable, and it follows that the map
$\theta:\ \D_{A,B}^{-1}\to A^\vee,\ x\mapsto \tr(x())$ is an isomorphism
of (divisorial) $A$-modules.

\begin{prp}\label{noeth_norm_domain_galois}
Let $A$ be a noetherian normal domain and
$\frak{G}\le {\rm Aut}(A)$ a finite group of ring automorphisms with ring of invariants
$B:=A^\frak{G}$. Then the following are equivalent:
\begin{enumerate}
\item $B\le A$ is a Galois-extension;
\item $_BA$ is projective and $A^\vee:=\Hom(_BA,{_BB})=A \cdot \tr_\frak{G}$.
\item $_BA$ is projective and $\D_{A,B}=A$ (or $\D_{A,B}^{-1}=A$).
\end{enumerate}
\end{prp}
\begin{proof}
``(1) $\Rightarrow$ (2)":\  By assumption $\rho:\ \Delta\to \End(_BA)$ is an isomorphism.
For any $\lambda\in A^\vee$ we have $\lambda=\rho(d)$ with
$d:=\sum_{g\in \frak{G}}a_g\cdot d_g\in \Delta.$
Then $\lambda(a)=\sum_{g\in \frak{G}}a_g g(a)\in A^\frak{G}$, hence for every $h\in \frak{G}$,
$\sum_{g\in \frak{G}}h(a_g)hg(a)=\sum_{g\in \frak{G}}a_g g(a)$, which implies
$$\sum_{g\in \frak{G}}h(a_g)d_{hg}=\sum_{g\in \frak{G}}a_g d_g\in\Delta$$ and therefore
$h(a_1)=a_h$. We get $d=\sum_{g\in \frak{G}}g(a_1)d_g=$
$\sum_{g\in \frak{G}}d_g\cdot a_1=\frak{G}^+\cdot a_1\in \frak{G}^+\cdot A.$
Hence $A^\vee\subseteq \rho(\frak{G}^+\cdot A)=A\cdot \tr_\frak{G}\subseteq A^\vee$.\\
``(2) $\Rightarrow$ (1) ":\ Since the field extension $\mathbb{L}\ge \mathbb{K}=\mathbb{L}^\frak{G}$ is
Galois, the map
$$\rho\otimes_B\mathbb{K}:\ \Delta\otimes_B \mathbb{K}\to
\End(_BA)\otimes_B \mathbb{K}=\End(_{\mathbb{K}}\mathbb{L})$$
is an isomorphism, so $\rho$ is injective. Since $_BA$ is finitely generated
and projective, the map
$\gamma^\vee:\ A\otimes_B\ A^\vee\to \End(_BA),\ a\otimes\lambda\mapsto a\cdot\lambda()=(a'\mapsto a\cdot \lambda(a'))$ is surjective (and bijective).
Hence $\rho(\Delta)\supseteq \rho(A\frak{G}^+A)=$
$\gamma^\vee(A\otimes_B \frak{G}^+A)=\End(_BA)$, so $\rho$ is surjective and therefore bijective.\\
``(2) $\iff$ (3)":\
Consider the isomorphism $\theta:\ \D_{A,B}^{-1}\to A^\vee,\ x\mapsto \tr(x())$. Then
$A^\vee=A\cdot \tr_\frak{G}$ if and only if for every
$x\in {\mathcal D}_{A,B}^{-1}$
there is $a\in A$ with $\theta(x)=\theta(a)$, i.e.
${\mathcal D}_{A,B}^{-1}\subseteq A$, which is equivalent to
${\mathcal D}_{A,B}^{-1}=A$ (since $A\subseteq {\mathcal D}_{A,B}^{-1}$
is always true) and equivalent to ${\mathcal D}_{A,B}=A$.
\end{proof}

\begin{cor}\label{intro_cor_2}
Let $\fld$ be algebraically closed and $X$ be a normal irreducible $\fld$-variety (so $A:=\fld[X]$
is a normal domain). Then the following are equivalent:
\begin{enumerate}
\item $\frak{G}$ acts freely on $X$;
\item $_{A^\frak{G}}A$ is projective and $\D_{A,A^\frak{G}}=A$ (or $\D_{A,A^\frak{G}}^{-1}=A$).
\end{enumerate}
\end{cor}

In the rest of this section we will recapitulate notation and results from earlier papers,
which will be used in the sequel.
For a finitely generated commutative $\fld$-algebra $A$ we will denote by $\Dim(A)$ the \emph{Krull-dimension}
of $A$. For a $\fld$-vector space $V$ we will denote with $\dim(V)=\dim_\fld(V)$ the $k$-dimension of that space.
So $\Dim(A)=0$ $\iff$ $\dim(A)<\infty$.

\begin{df}\label{point_df}
Let $A\in\ts$, then an element $a\in A$ with $\tr(a)=1$ is called a {\bf point} in $A$.
\end{df}

In \cite{nonlin} Theorem 4.1 and Proposition 4.2 the following general result has been shown:
\begin{thm}\label{first_main}
Let $A$ be trace-surjective and $a\in A$ be a point, then:
\begin{enumerate}
\item $A=\oplus_{g\in G}A^G\cdot (a)g$ is a free $A^G$-module with basis $\{(a)g\ |\ g\in G\}$
and also a free $A^G[G]$ module of rank one, where $A^G[G]$ denotes the group ring of $G$ over $A^G$.
\item If $S:=\fld [(a)g\ |\ g\in G]\le A$ is the subalgebra generated by the $G$-orbit of the point $a$,
then $A=A^G\otimes_{S^G}S$.
\end{enumerate}
\end{thm}

Now let $V=V_{reg}$ and $V^*:=\oplus_{g\in G} \fld X_g\cong \fld G$, with $X_g=(X_{1_G})g$, be
the regular representation of $G$ and set $S_{reg}:={\rm Sym}(V^*)$ (note that $V^*$
and $V$ are isomorphic $\fld G$-modules). Set $X:=\sum_{g\in G} X_g\in B:=S_{reg}^G$,
then $V^{*G}=\fld \cdot X$. Following \cite{nonlin} Definition 2, we set
$D_\fld := D_\fld(G):=S_{reg}/(\alpha)$ with $\alpha = X-1.$
Then $D_\fld \cong \fld [x_g\ |\ 1\ne g\in G]$, with $x_g:=\overline X_g$ and
${\rm tr}(x_1)=1$, is a polynomial ring of Krull dimension $|G|-1$ and there is an isomorphism
of trace-surjective $\fld -G$-algebras $D_\fld \cong (S_{reg}[1/X])_0;\ x_g\mapsto X_g/X.$
Moreover there is an isomorphism of $\mathbb{Z}$-graded trace-surjective algebras:
$D_\fld [X,1/X]\to \sum_{z\in\mathbb{Z}} D_\fld  X^z=S_{reg}[1/X].$
Taking $G$-invariants on both sides we obtain an isomorphism of
$\mathbb{Z}$-graded $\fld $-algebras:
$D_\fld^G[X,1/X]\cong S_{reg}^G[1/X].$ As mentioned in Theorem \ref{arb_p_grp_sec0},
there is a retract $U\le D_\fld$ with $U\in\ts$ such that the rings $U$, $U^G$ and
$D_\fld^G$ are polynomial rings.  We will show that the algebras $D_\fld$ and $U$ are $s$-projective
\footnote{with respect to surjective functions rather than epimorphisms} objects in $\mathfrak{Ts}$
(see Definition \ref{universal_df} and Theorem \ref{standard_is_cycl_prj}).
It has been shown in \cite{nonlin} Proposition 5.5. that the Krull-dimension $\mathrm{log}_p(|G|)$ of $U$
is the minimal possible number of generators for a trace-surjective subalgebra of $D_\fld$, if $k = \mathbb{F}_p$.

\section{Universal, projective and generating objects in the category $\ts$}\label{sec_s-Projectivity}

From now on, unless explicitly stated otherwise, $G$ will denote a \emph{non-trivial} finite
$p$-group.

The category $\ts$ is non-abelian but it has finite
coproducts, given by tensor products over $\fld$. This together
with the structure theorem \ref{first_main} gives rise to the concepts
of weakly initial, generating, projective and free objects, in analogy to module categories.
In particular there are categorical characterizations of $D_\fld$ and its standard subalgebras in $\ts$,
as defined in \cite{nonlin} Definition 3, comparable to projective generators in module categories,
which we are now going to develop. This was announced in \cite{nonlin} Remark 5.

Let $\C$ be an arbitrary category. Then an object $u\in\C$ is called \emph{weakly initial},
if for every object $c\in\C$ the set $\C(u,c):=\Mor_\C(u,c)$ is not empty, i.e.
if for every object in $\C$ there is at least one morphism from $u$ to that object.
If moreover $|\C(u,c)|=1$ for every $c\in\C$, then $u$ is called an \emph{initial object}
and is uniquely determined up to isomorphism.

An object $m \in \C$ is called a \emph{generator} in $\C$, if the covariant morphism - functor
${\rm Mor}_\C(m,*)$ is injective on morphism sets. In other words, $m$ is a
generator if for any two objects $x,y \in \C$ and morphisms
$f_1,\ f_2 \in \C(x,y)$, $f_1 \ne f_2$ implies $(f_1 )_* \ne (f_2)_*$, i.e.
there is $f \in \C(m,x)$ with $f_1\circ f \ne f_2 \circ f$. It follows that
$\C(m,x)\ne\emptyset$ whenever $x\in\C$ has nontrivial automorphisms. So if every
object $x\in\C$ has a nontrivial automorphism, then generators in $\C$ are weakly initial objects.
If $\C=\ts_G$ then right multiplication with any $1\ne z\in Z(G)$ is a nontrivial automorphism for every $A\in\ts$, hence every generator in $\ts_G$ is weakly initial.

Recall that in an arbitrary category $\C$ an object $x$ is called ``projective" if the covariant
representation functor $\C(x,?):=\Mor_\C(x,?)$ transforms epimorphisms into
surjective maps. If $\C$ is the module category of a ring, then a morphism
is an epimorphism if and only if it is surjective. Therefore a module $M$
can be defined to be projective, if $\Mor_\C(M,?)$ turns surjective morphisms to
surjective maps. In the category $\ts$, however, there are
non-surjective epimorphisms (e.g. $A^p\hookrightarrow A$ for a domain $A\in\ts$).
This leads to the slightly modified notions of ``s-generators" and ``s-projective objects"
in the category $\ts$:

\begin{df}\label{universal_df}
Let $B$ be a $\fld-G$ algebra in $\ts$.
\begin{enumerate}
\item $B$ is called  {\bf universal}, if it is a weakly initial object in $\ts$.
\item $\Gamma\in\ts$ is an {\bf s-generator} if for every $R\in\ts$ there is a surjective
morphism $\Psi:\ \Gamma^{\otimes\ell}\to R$ for some $\ell\ge 1$.
\item $A\in\ts$ is called {\bf s-projective}, if the covariant representation
functor $\ts(A,*)$ transforms {\bf surjective morphisms} into surjective maps.
\end{enumerate}
\end{df}

Let $a\in A$ be a point, i.e. $\tr(a)=1$.  Then the map $X_g\mapsto (a)g$ for $g\in G$ extends
to a $k$-algebra homomorphism ${\rm Sym}(V_{reg}^*)\to A$ with $\alpha\mapsto 0$,
hence it defines a unique morphism
$\phi:\ D_\fld\to A$ with $\phi\in\ts$, mapping $x_g\mapsto (a)g$. In other words
$D_\fld$ has a ``free point" $x_e$, which can be mapped to any point $a\in A\in \ts$
to define a morphism $\phi\in \ts(D_\fld,A)$. It is not hard to see that, due to the existence of
these free points $x_g$, the algebra $D_\fld$ is $s$-projective in $\ts$.
The following generalization has been shown in \cite{universal}:

\begin{thm}[\cite{universal} Theorem 2.8]\label{tr_id_standard}
Let $W\to V$ be an epimorphism of finite dimensional $\fld G$-modules,
$S:={\rm Sym}(V^*)\hookrightarrow T:={\rm Sym}(W^*)$ the
corresponding embedding of symmetric algebras and $v^*\in (V^*)^G$. Assume that
$\bar S:=S/(v^*-1)S$ is in $\ts$.
Then $\bar S$ is a retract of $\bar T:=T/(v^*-1)T$ and $\bar T$ and $\bar S$ are
$s$-projective objects in $\ts$.
\end{thm}

\begin{rem}\label{universal_df_rem}
\begin{enumerate}
\item It is easy to see that every s-generator and every s-projective object is also universal.
\item Every $A\in\ts$ with $\ts(A,P)\ne\emptyset$ for some s-projective $P\in\ts$ is universal.
So the universal objects are precisely the objects of $\ts$ that map to $D_\fld$.
\item The commutative artinian ``diagonal group ring" $\fld G:=\oplus_{g\in G}\fld e_g$
with $e_ge_h=\delta_{g,h}e_g$ and regular $G$-action is a non-universal object in $\ts$.
\end{enumerate}
\end{rem}

The following Lemma characterizes universal objects in $\ts$ and also indicates the particular
significance of this notion in that category:

\begin{lemma}\label{char_universal}
Let $\frak W\in \ts$, then the following are equivalent:
\begin{enumerate}
\item $\frak W$ is universal;
\item $\frak W/I\le D_\fld$ for some $G$-stable prime ideal $I\le\frak W$;
\item every $A\in\ts$ can be written as $A\cong A^G\otimes_{S^G}S$ where $S\le A$
is a subalgebra isomorphic to $\frak W/I$ for some $G$-stable ideal $I\unlhd \frak W$.
\item every $A\in\ts$ is of the form $A\cong R\otimes_{{\frak W}^G}\frak{W}$
for some $\fld$-algebra $R$ with trivial $G$-action and homomorphism
${\frak W}^G\to R$.
\end{enumerate}
\end{lemma}

\begin{proof}
(1) $\iff$ (2):\ This has been shown above. \\
(1) $\Rightarrow$ (3):\
Let $\phi\in\ts(\frak W,A)$ with $S:=\phi(\frak W)\le A$,
then $S\cong \overline{\frak W}:=\frak W/I$ for some $G$-stable ideal $I\unlhd \frak W$
and it follows from Theorem \ref{first_main}, that
$A\cong A^G\otimes_{\bar{\frak W}^G}\bar{\frak W}.$\\
(3) $\Rightarrow$ (1):\ This follows from
``(1) $\iff$ (2)" and choosing $A=D_\fld.$
Finally, (3) and (4) are different ways of expressing the same situation.
\end{proof}

Let $R\in\ts$ with point $w\in R$, $w_g:=(w)g$ for $g\in G$ and
$R^G=\fld[r_1,\cdots,r_n]$, then by Theorem \ref{first_main},
$$R=\fld[R^G,w_g\ |\ g\in G]=\fld[w_g,w_g+r_i\ |\ g\in G, i=1,\cdots,n]$$
with $\tr(w_g+r_i)=\tr(w_g)+|G|r_i=1$ for all $g\in G$ and $i=1,\cdots,n$ (since $G\ne 1$).
So $R=\fld[v_1,\cdots,v_\ell]$, with points $v_i$ so we conclude:

\begin{lemma}\label{RTs_gen_by_points}
Every object $R\in\ts$ is generated by a finite set of points.
\end{lemma}

Recall that the finite coproducts in $\ts$ are given by the tensor-product over $\fld$.
A finite tensor product of $\fld$-$G$ algebras
lies in $\ts$ if at least one of the factors does. In particular the category $\ts$ also has finite coproducts given by the tensor-product over $\fld$. Recall that for an object $A\in\ts$ and $\ell\in\mathbb{N}$
we define
$$A^{\otimes\ell}:=\coprod_{i=1}^\ell A:= A\otimes_\fld A\otimes_\fld\cdots\otimes_\fld A$$
with $\ell$ copies of $A$ involved. This allows for the following partial characterization
of \emph{categorical generators} in $\ts$:

\begin{lemma}\label{generators in ts}
If $\Gamma$ is an $s$-generator, then it is a categorical generator in $\ts$.
\end{lemma}
\begin{proof}
Let $\alpha,\beta\in\ts(R,S)$ with $\alpha\circ\psi=\beta\circ\psi$ for
all $\psi\in\ts(\Gamma,R)$.
By assumption we have the following commutative diagram
\begin{diagram}\label{Delta_diagram1}
\Gamma^{\otimes\ell} &\rOnto^\Psi          &R \\
 \uTo^{\tau_i}     &\ruTo_{\Psi_i}       & \\
\Gamma             &                     & \\
\end{diagram}
where $\tau_i$ maps $\gamma\in \Gamma$ to $1\otimes\cdots\otimes \gamma\otimes\cdots\otimes 1$
and $\Psi$ is surjective. Then $\alpha\circ\Psi\circ\tau_i=\beta\circ\Psi\circ\tau_i$ for
all $i$, hence $\alpha\circ\Psi=\beta\circ\Psi$. Since $\Psi$ is surjective
it follows that $\alpha=\beta$, so $\Gamma$ is a generator in $\ts$.
\end{proof}

We will now give some definitions that turn out to be useful in finding criteria for s-projectivity
and the s-generator property:
\begin{df}\label{triang_dissolv}
Let $\E$ be an $\fld$-$G$-algebra of Krull dimension $N$.
\begin{enumerate}
\item $\E$ is said to be {\bf erasable}, if for every $A\in\ts$, the tensor product
$A\otimes_\fld\E$ erases the $G$-action on $\E$ in the sense that
\[
A\otimes_\fld\E= (A\otimes_\fld 1)[\lambda_1,\cdots,\lambda_N]\cong A^{[N]},
\]
with the isomorphism being the identity on $A$ and $\fld^{[N]}\cong\fld[\lambda_1,\cdots,\lambda_N]\subseteq (A\otimes_\fld\E)^G$.
\item If $\E\in\ts$ and isomorphism in (1) holds for $A=\E$, then
$\E$ is called {\bf self-erasing}.
\end{enumerate}
\end{df}

\begin{prp}\label{triangular implies erasable}
Let $\Gamma\in\ts$ be a polynomial ring with triangular $G$-action.
Then $\Gamma$ is erasable.
\end{prp}
\begin{proof}
We assume that $\Gamma=\fld[T_1,\cdots,T_N]\in\ts$ is a polynomial ring such that
for each $g\in G$ and $1\le i\le N$ we have $(T_i)g=T_i+f_{i,g}$ with
$f_{i,g}\in\fld[T_1,\cdots,T_{i-1}]$. Now let $A\in\ts$ and $a\in A$ with
$\tr(a)=1$. Then $\tr(aT_i)=\sum_{g\in G}(a)g\cdot (T_i)g=$
$\sum_{g\in G}(a)g\cdot (T_i+f_{i,g})=$
$\tr(a)\cdot T_i+\sum_{g\in G}(a)g\cdot f_{i,g}.$ Hence
$T_i-\tr(aT_i)\in A[T_1,\cdots,T_{i-1}]$. Therefore an obvious induction argument
shows that
$$A\otimes_\fld\Gamma=A[T_1,\cdots,T_N]=A[\tr(aT_1),\cdots,\tr(aT_N)],$$
so $\Gamma$ is erasable.
\end{proof}

\begin{prp}\label{erasable_inv_is_poly_intersetc_poly}
Let $\E$ be an erasable $\fld-G$-algebra of Krull-dimension $e$ (not necessarily in $\ts$) and
let $\mathbb{P}\in\ts$. Then the following hold:
\begin{enumerate}
\item $\E^G=\E\cap (\E\otimes_\fld \mathbb{P})^G$ with
$(\E\otimes_\fld \mathbb{P})^G\cong (\mathbb{P}^{[e]})^G\cong (\mathbb{P}^G)^{[e]}.$
\item If $\mathbb{P}^G\cong \fld^{[m]}$, then $(\E\otimes_\fld \mathbb{P})^G\cong \fld^{[e+m]}$.
\item If $\mathbb{P}$ is s-projective, then so is $\E\otimes_\fld\mathbb{P}$.
\end{enumerate}
\end{prp}
\begin{proof}
(1) and (2):\ Clearly $\E^G=\E\cap (\E\otimes_\fld \mathbb{P})^G$.
By definition of ``erasable", $\F :=\E\otimes_\fld \mathbb{P}\cong
\mathbb{P}\otimes_\fld\fld[T_1,\cdots,T_n]\cong\mathbb{P}^{[e]}$
with $\fld[\underline T]\le \F^G$, hence $(\mathbb{P}^{[e]})^G\cong (\mathbb{P}^G)^{[e]}$. \\
(3):\ Let $\alpha:\ A\to B\in \ts$ be surjective and $\beta:\ \F =\mathbb{P}[T_1,\cdots,T_e]\to B$ be
morphisms in $\ts$. Choose ${\bf a}:=(a_1,\cdots,a_e)\in A^e$
with $\alpha(a_i)=\beta(T_i)$ and $\theta\in\ts(\mathbb{P},A)$ with
$\alpha\theta=\beta_{|\mathbb{P}}$. Then $\theta$ extends to a map
$$\tilde\theta:\ \F\to A,\ \sum_{\mu\in\mathbb{N}_0^e}p_\mu\underline T^\mu\mapsto
\theta(p_\mu)a_1^{\mu_1}\cdots a_e^{\mu_e}$$
with $\alpha\circ\tilde\theta=\beta$. Since the $T_i$ are $G$-invariant,
$\tilde\theta\in \ts(\F,A)$, which shows that $\F$ is s-projective.
\end{proof}

\begin{thm}\label{erasbility_theorem}
Let $\Gamma\in\ts$. Then $\Gamma$ is erasable if and only if
$\Gamma$ is self-erasing and any one of the following equivalent
conditions is satisfied:
\begin{enumerate}
\item $\Gamma$ is universal;
\item $\Gamma$ is $s$-projective;
\item $\Gamma$ is an $s$-generator.
\end{enumerate}
\end{thm}

\begin{proof} ``Only if":\  Suppose that $\Gamma$ is erasable, then clearly $\Gamma$ is
self-erasing (take $A=\Gamma$). Now put $A=D_\fld$; then $D_\fld\otimes_\fld\Gamma=D_\fld[\lambda_1,\cdots,\lambda_d]=:D_\fld[\underline\lambda]$, where
each $\lambda_i$ is invariant and $d$ is the Krull-dimension of $\Gamma$. Now
$D_\fld$ is s-projective and so $D_\fld[\underline\lambda]$ is also s-projective
by Proposition \ref{erasable_inv_is_poly_intersetc_poly} (3).
Further,
as $D_\fld$ is universal, there exists a morphism $D_\fld\to\Gamma$. Therefore $\Gamma$ is a
direct summand of $D_\fld[\underline\lambda]$ and so is s-projective and hence universal.
Finally, as $\Gamma$ is self-erasing, $\Gamma\otimes_\fld\Gamma=\Gamma[\underline\mu]:=\Gamma[\mu_1,\cdots,\mu_d]$ where
each $\mu_i$ is invariant. Hence, by a simple induction argument,
$\Gamma^{\otimes (m+1)}\cong\Gamma[\nu_1,\cdots,\nu_{md}]$ for all $m\ge 1$. Since $\Gamma$
is universal it follows easily from Lemma \ref{RTs_gen_by_points}, that every $A\in\ts$
is surjective image of some $\Gamma^{\otimes\ell}$, so
$\Gamma$ is also an s-generator.\\
``If":\  Note first that if $\Gamma$ is either s-projective or an s-generator,
then $\Gamma$ is universal. Now suppose that $\Gamma$ is self-erasing and universal and let
$A\in\ts$. Then there exists a morphism $\theta:\ \Gamma\to A$. Now, as above,
$\Gamma\otimes_\fld\Gamma=\Gamma[\underline\mu]$ with invariants
$\underline\mu=(\mu_1,\cdots,\mu_d)$ and further $A\cong A^G\otimes_{\Gamma^G}\Gamma$
where $\Gamma^G\to A^G$ is induced by $\theta$. Hence
$$A\otimes_\fld\Gamma\cong (A^G\otimes_{\Gamma^G}\Gamma)\otimes_\fld\Gamma\cong
A^G\otimes_{\Gamma^G} (\Gamma\otimes_\fld\Gamma)\cong A^G\otimes_{\Gamma^G} (\Gamma[\underline\mu])\cong
(A^G\otimes_{\Gamma^G} \Gamma)[\underline\mu]\cong
A[\underline\mu]$$
with trivial $G$-action on $\fld[\underline\mu]$.
Thus $\Gamma$ is erasable.
\end{proof}

\begin{cor}\label{D_k_is_pro_gen}
The following algebras in $\ts$ are triangular polynomial rings and therefore
erasable s-projective generators:
\begin{enumerate}
\item Every algebra $\bar S:=S/(v^*-1)S \in\ts$ as in Theorem \ref{tr_id_standard}.
\item The algebra $D_\fld$ and its standard retract $U$ (as in Theorem \ref{arb_p_grp_sec0}).
\end{enumerate}
\end{cor}
\begin{proof} The proof of Theorem \ref{arb_p_grp_sec0}, given in \cite{nonlin} shows
that $U$ is a polynomial ring on which $G$ acts in a triangular way. All the other
algebras are visibly triangular, so the claim follows from
Proposition \ref{triangular implies erasable} and Theorem \ref{erasbility_theorem}.
\end{proof}

\begin{thm}\label{erasable_is_poly_bd_stably_inv}
Let $\Gamma\in\ts$ of Krull dimension $d$ and assume that $\Gamma$ is erasable.
Then $\Gamma\cong\fld^{[d]}$ and $\Gamma^G\otimes_\fld\fld^{[n]}\cong \fld^{[n+d]}$ with
$n=\log_p(|G|)$. Moreover
\begin{enumerate}
\item $\Gamma$ is $n$-stably triangular.
\item $\Gamma^G$ is $n$-stably polynomial.
\end{enumerate}
\end{thm}
\begin{proof} We use notation of the proof of Theorem \ref{erasbility_theorem}.
Suppose that $\Gamma\in\ts$ is erasable. Then $\Gamma$ is universal and so there
is a morphism $\Gamma\to \fld G$, the disconnected abelian group algebra
from Remark \ref{universal_df_rem}. Hence $\Gamma$ has a maximal ideal $\frak{m}$ with
$\Gamma/\frak{m}\cong\fld$. Since $\Gamma$ is self-erasing, $\Gamma\otimes_\fld\Gamma=
\Gamma[\underline\mu]$, therefore $\Gamma\cong\fld\otimes_\fld\Gamma=$
$\Gamma/\frak{m}\otimes_\fld\Gamma=(\Gamma\otimes_\fld\Gamma)/\frak{m}^e=
\Gamma[\underline\mu]/\frak{m}^e=\Gamma/\frak{m}[\underline\mu]\cong\fld^{[d]}$
as $\fld$-algebras, with $\fld^{[d]}\cong\fld[\underline\mu]\subseteq
(\Gamma\otimes_\fld\Gamma)^G$.
Here $\frak{m}^e$ denotes the extended ideal in $\Gamma\otimes_\fld\Gamma$.\\
Now by Proposition \ref{triangular implies erasable} the algebra $U\cong\fld^{[n]}$
is erasable, therefore, as before, there are invariants $\underline\lambda=(\lambda_1,\cdots,\lambda_d)$
and $\underline\alpha=(\alpha_1,\cdots,\alpha_n)$ such that
$U[\underline\lambda]=U\otimes_\fld\Gamma\cong\Gamma\otimes_\fld U\cong\Gamma[\underline\alpha]$
is triangular. Hence
$$\Gamma^G[\underline\alpha]=(\Gamma[\underline\alpha])^G=
(U[\underline\lambda])^G=U^G[\underline\lambda]\cong\fld^{[n+d]}$$
and so $\Gamma^G$ is $n$-stably polynomial.
Since $U$ is triangular, $\Gamma$ is $n$-stably triangular.
\end{proof}

\begin{cor}\label{erasable_is_poly_bd_stably_inv_cor1}
Let $\Gamma\in\ts$ be erasable and assume that $\Gamma^G\cong\fld^{[d]}$, then
$(\Gamma^{\otimes\ell})^G\cong \fld^{[d\ell]}$. In particular this
is satisfied for $\Gamma\in\{D_\fld,U\}$.
\end{cor}
\begin{proof} The proof is by induction on $\ell$. For $\ell=1$ the statement is true
by the hypothesis.
$\Gamma^{\otimes\ell}\cong\Gamma^{\otimes(\ell-1)}\otimes_\fld\Gamma\cong$
$\Gamma^{\otimes(\ell-1)}\otimes_\fld\fld[\lambda_1,\cdots,\lambda_d]$ with
$\fld[\underline\lambda]\le (\Gamma^{\otimes\ell})^G$. Hence
$$(\Gamma^{\otimes\ell})^G\cong (\Gamma^{\otimes(\ell-1)}\otimes_\fld \fld[\underline\lambda])^G\cong
(\Gamma^{\otimes(\ell-1)})^G\otimes_\fld \fld[\underline\lambda]\cong
\fld^{[d(\ell-1)]}\otimes_\fld \fld^{[d]}\cong
\fld^{[d\ell]}.$$
\end{proof}

As a consequence we see that every stably triangular $\fld-G$ algebra is a
tensor factor of a tensor power of $U$ or of $D_\fld$:

\begin{cor}\label{erasable_is_poly_bd_stably_inv_cor4}
Let $X\in\{D_\fld,U\}$ and let $A$ be a stably triangular $\fld-G$ algebra.
Then there is an erasable algebra $B\in\ts$ and an $N\in\mathbb{N}$,
such that $A\otimes_\fld B\cong X^{\otimes N}$.
If moreover $A\in\ts$, then $A^G$ is stably polynomial.
\end{cor}
\begin{proof}
First note that, since $U$ is erasable, we have for every $\ell\in\mathbb{N}$:
$U^{\otimes\ell}\cong U\otimes_\fld U^{\otimes(\ell-1)}\cong
U[\lambda_1,\cdots,\lambda_{n(\ell-1)}]\cong U^{[n(\ell-1)]}$ with
$\fld[\lambda_1,\cdots,\lambda_{n(\ell-1)}]\le (U^{\otimes\ell})^G$.
Similarly $D_\fld^{\otimes\ell}\cong U^{[(|G|-1)\ell-n]}$.
Assume now that $A\otimes_\fld F\cong\fld^{[N]}$ is triangular with $F=F^G\cong\fld^{[m]}$. Then $A\otimes_\fld F$ is erasable by Proposition \ref{triangular implies erasable}.
Hence $A\otimes_\fld F\otimes_\fld U\cong U\otimes_\fld \fld[\underline\beta]\cong U^{[N]}$ with
$\fld^{[N]}\cong\fld[\underline\beta]\le (A\otimes_\fld F\otimes_\fld U)^G$.
Now let $\ell\in\mathbb{N}$ be minimal
with $\ell>N/n+1$ and set $M:=(\ell-1)\cdot n-N>0$. Then with
$B:=F\otimes_\fld U^{[M]}$ we obtain
$A\otimes_\fld B\cong U^{[N]}\otimes_\fld \fld^{[M]}\cong
U^{[N+M]} \cong U^{[(\ell-1)\cdot n]}\cong U^{\otimes\ell}$.
Similarly let $\ell'\in\mathbb{N}$ be minimal
with $\ell'>\frac{N+n}{|G|-1}$ and set $M':=\ell'(|G|-1)-n-N>0$. Then with
$B':=F\otimes_\fld U^{[M']}$ we obtain
$A\otimes_\fld B'\cong U^{[N]}\otimes_\fld \fld^{[M']}\cong U^{[N+M']}\cong U^{[\ell'(|G|-1)-n]}\cong D_\fld^{\otimes\ell'}$.
Now assume in addition that $A\in\ts$. Since $B\in\ts$ is erasable, $A\otimes_\fld B\cong A\otimes_\fld\fld[\underline\alpha]$ with $\fld[\underline\alpha]\le$
$(A\otimes_\fld B)^G=A^G\otimes_\fld\fld[\underline\alpha]\cong
(U^{\otimes\ell})^G$, which is a polynomial ring by Theorem \ref{cat_gal_main1}.
It follows that $A^G$ is stably polynomial.
\end{proof}

\noindent
We now give a {\bf Proof of Theorem \ref{cat_gal_main2}}:
\begin{thm}\label{cat_gal_main2_pf}
Let $\Pp\in\ts$ be s-projective, then both, $\Pp$ and $\Pp^G$ are retracts of polynomial
rings over $\fld$.
\end{thm}
\begin{proof}
Let $\Pp\in\ts$ be $s$-projective and $\Gamma=U$ or $D_\fld$, then there is a surjective
morphism $\Gamma^{\otimes \ell}\to\Pp$, which splits, since $\Pp$ is s-projective.
It follows that $\Pp$ and $\Pp^G$ are retracts of
$\Gamma^{\otimes \ell}$ and $(\Gamma^{\otimes \ell})^G$, respectively.
Both of these rings are polynomial rings over $\fld$.
\end{proof}

We now give a {\bf Proof of Theorem \ref{SVG_is_poly_inters_poly_intro}} from the introduction:
\begin{thm}\label{SVG_is_poly_inters_poly_pf}
Let $S\cong \fld^{[d]}$ be a polynomial ring with triangular $G$-action (e.g. $S=S(V^*)$).
Then the ring of invariants $S^G$ is the intersection of two polynomial subrings inside an
$s$-projective polynomial $\fld-G$-algebra $\fld^{[N]}\in\ts$ of Krull-dimension $N=d+n$ with
$n:=\log_p{|G|}$. If moreover $S\in\ts$, then $S^G\otimes_\fld\fld^{[n]}\cong\fld^{[n+d]}$,
i.e. $S^G$ is $n$-stably polynomial.
\end{thm}
\begin{proof}
Clearly $S(V^*)$ is a triangular $\fld-G$-algebra (taking a triangular
basis for the vector space $V^*$) and $U$ is triangular by Theorem \ref{arb_p_grp_sec0}.
If follows from Proposition \ref{triangular implies erasable}
that $S$ and $U$ are erasable. Now the first claim follows from Proposition \ref{erasable_inv_is_poly_intersetc_poly}, taking $\mathbb{P}=U$ with $U^G\cong\fld^{[n]}$
and  $n=\log_p{|G|}$, by Theorem \ref{arb_p_grp_sec0}.\\
If in addition $S\in\ts$, then $S\otimes_\fld U\cong S\otimes_\fld \fld[\lambda_1,\cdots,\lambda_n]$
with $\fld^{[n]}\cong \fld[\lambda_1,\cdots,\lambda_n]\le (S\otimes_\fld U)^G$, by Theorem
\ref{cat_gal_main1}(2). It follows, switching the roles of $U$ and $S$, that
$S^G\otimes_\fld\fld^{[n]}\cong (S\otimes_\fld U)^G\cong (U\otimes_\fld S)^G\cong
(U^G)^{[d]}\cong\fld^{[n+d]}$.
\end{proof}

\begin{df}\label{standard_cyclic_df}
Following \cite{nonlin} Definition 3, we call a
trace-surjective $G$-subalgebra $S\le D_\fld$
{\bf standard}, if it is a \emph{retract} of $D_\fld$, or in other words, if $D_\fld=S\oplus J$,
where $J$ is some $G$-stable (prime) ideal.
We also call $C\in\ts$ {\bf cyclic}, if $C\cong D_\fld/I$ with $G$-stable ideal $I$.
Equivalently, $C$ is generated by one $G$-orbit of a point $c\in C$.
\end{df}

In this terminology, $U$ is standard as well as cyclic and s-projective.
The next theorem shows that the latter two properties characterize standard subalgebras of $D_\fld$:
\\[1mm]
\begin{thm}\label{standard_is_cycl_prj}
Let $R,S\in\ts$, then the following hold:
\begin{enumerate}
\item $R$ is s-projective, if and only if it is retract of a tensor product $D_\fld^{\otimes\ell}$.
\item $S$ is a standard subalgebra of $D_\fld$ if and only if $S\in\ts$ is cyclic and s-projective. \end{enumerate}
\end{thm}
\begin{proof} (1): This follows from the fact that every $R\in \ts$ is surjective image of some $D_\fld^{\otimes\ell}$.\\
(2): By definition every standard subalgebra is cyclic and
a retract of $D_\fld$. Hence by (1) it is $s$-projective.
Now let $S\in\ts$ be cyclic and $s$-projective.
Then there is a surjective morphism
$D_\fld\to S$, which must split, hence $S$ is a standard subalgebra.
\end{proof}

\begin{rem}\label{standard_is_cycl_prj_rem}
The statement (1) in Theorem \ref{standard_is_cycl_prj} shows
that the algebra $D_\fld^{\otimes\ell}$ is the analogue in $\ts$ of the free module of rank $\ell$ in a module category.
\end{rem}

\section{Basic Algebras}\label{sec_Minimal universal objects}

Let $\C$ be an arbitrary category, then for objects $a,b\in\C$ one defines $a\prec b$ to mean that there is a monomorphism
$a\hookrightarrow b\in\C$ and $a\thickapprox b$ if $a\prec b$ and $b\prec a$.
According to this definition, an object $b\in\C$ is called \emph{minimal} if
$a\prec b$ for $a\in\C$ implies $b\prec a$ and therefore $a\thickapprox b$.
Clearly ``$\thickapprox$" is an equivalence relation on the object class of $\C$.
Recall that $A\in\ts$ is universal if it is weakly initial, or, equivalently, if it maps
to $D_\fld$.

\begin{df}\label{basic_df}
The algebra $B\in\ts$ is called {\bf basic} if it is universal and minimal.
\end{df}

The following Lemma characterizes types of morphisms in $\ts$ by their action on points.
The results will then be used to analyze basic objects in $\ts$:

\begin{lemma}\label{surj_surjects_points}
A morphism $\theta\in \ts(R,S)$  is surjective (injective, bijective) if and only if it induces a surjective (injective, bijective) map from the set of points of $R$ to the set of points of $S$.
In particular $\theta\in \ts(R,S)$ is a monomorphism if and only if $\theta$ is injective.
\end{lemma}
\begin{proof} ``Surjectivity": Let $s\in S$ with $\tr(s)=1$ and $r\in R$ with $\theta(r)=s$. Then $r':=\tr(r)-1\in \ker(\theta)\cap R^G$. Let $w\in R$ with $\tr(w)=1$, then $r'=\tr(r'w)$ and $v:=r-r'w$
satisfies $\theta(v)=s$ and $\tr(v)=1$, hence the induced map on points is surjective.
On the other hand, since $R$ and $S$ are generated as algebras by points, the reverse conclusion follows.\\
``Injectivity": We can assume that the induced mapping on points is injective and want to show that
$\theta$ is injective. Let $w\in R$ be a point and $r,r'\in R^G$ with $\theta(r)=\theta(r'),$
then $\tr(r+w)=\tr(w)=1=\tr(r'+w)$ and $\theta(r+w)=\theta(r'+w)$, so $r+w=r'+w$ and
$r=r'$. Hence the induced map on the rings of invariants is injective.
But $R=\oplus_{i=1}^n R^G w_i$, with $n=|G|$ and a $G$-orbit of points $\{w_1,\cdots,w_n\}$.
It follows that $V':=\langle\theta(w_i)\ |\ i=1,\cdots,n\rangle\le S$ is a copy of the regular representation of $G$, so
by \ref{first_main} we have $S=\oplus_{i=1}^n S^G\theta(w_i)$. Let $r=\sum_{i=1}^n r_iw_i, r'=\sum_{i=1}^n r'_iw_i$
with $r_i,r'_i\in R^G$ and $\theta(r)=\theta(r')$, then
$\sum_{i=1}^n \theta(r_i-r_i')\theta(w_i)=0$ implies $\theta(r_i)=\theta(r'_i)$,
so $r_i=r'_i$ for all $i$ and therefore $r=r'$. \\
For the last claim, it is clear that an injective morphism is a monomorphism, so assume now
that $\theta$ is a monomorphism. It suffices to show that $\theta$ is injective on the points of $R$, so
let $a_1, a_2\in R$ be points with $\theta(a_1)=\theta(a_2)$. Define $\psi_i:\ D_\fld\to R$ as the morphisms
determined by the map $D_\fld\ni x_1\mapsto a_i$, then $\theta\circ\psi_1=\theta\circ\psi_2$, hence
$\psi_1=\psi_2$ and $a_1=a_2$. This finishes the proof.
\end{proof}

Now let $A$ be an object in $\ts$, then $\Dim(A)=\max\{\Dim(A/{\rm p})\ |\ {\rm p}\in {\rm Spec}(A)\}$,
where $\Dim(A/{\rm p})={\rm transc.deg}_\fld(A/{\rm p}):={\rm transc.deg}_\fld({\rm Quot}(A/{\rm p}))$,
the transcendence degree over $\fld$ of the quotient field ${\rm Quot}(A/{\rm p})$.
If $B\prec A$ then
$$\Dim(B)={\rm transc.deg}_\fld(B)\le{\rm transc.deg}_\fld(A)=\Dim(A).$$
This is clear if $A$ is a domain and an easy exercise otherwise.
In particular, any two $\thickapprox$-equivalent domains in $\ts$ have the same Krull-Dimension.

\noindent
If $A\in\ts$ is universal it maps into $D_\fld$ with a universal image isomorphic to $A/I$ for some
$G$-invariant prime ideal $I\unlhd A$. So every universal object has a quotient which is a universal
integral sub-domain of $D_\fld$. Notice also that if $B\prec A$ with universal $A$, then
$B$ is also universal; so if $A$ is minimal among the universal objects, then $A$ is also a minimal
object and therefore basic. It is however not completely obvious from the definition that basic objects
do exist. This is established as follows, which also shows the existence of basic normal domains:

\begin{lemma}\label{basics_exist}
Let $X\in\ts$ be a subalgebra of $U$ or of $D_\fld$ and let
$\hat X$ denote its normal closure in ${\rm Quot}(X)$. Then
$\hat X$ is universal in $\ts$. Moreover if $X$ is a subalgebra of minimal
Krull-dimension in $U$ or in $D_\fld$, then $X$ and
$\hat X$ are basic domains.
\end{lemma}

\begin{proof}
The polynomial rings $U$ and $D_\fld$ are universal domains of Krull-dimension $n$
and $|G|-1=p^n-1$, respectively. Let $X\in\ts$, $X\hookrightarrow U$ or $D_\fld$,
then $X$ is a universal domain. Now suppose that $X$ has minimal Krull-dimension.
If $Y\prec X$, then $\Dim(Y)=\Dim(X)$, but there is $\alpha\in\ts(X,Y)$ with
$\alpha(X)\prec Y\prec X$. It follows that $\Dim(\alpha(X))=\Dim(Y)=\Dim(X)$, so
$\ker(\alpha)=0$ and $X\prec Y$. This shows that $X$ is a universal minimal, hence basic, domain.\\
Since $X$ is a finitely generated $\fld$-algebra, so is $\hat X$ and, since $U$ and $D_\fld$
are normal rings, $\hat X\le U$ or $\hat X\le D_\fld$, respectively. It follows that
$\hat X$ is universal, and basic, if $X$ is.
\end{proof}

The next result describes properties of basic objects and shows
that they form a single $\thickapprox$-equivalence class consisting of integral domains,
all of which have the same Krull-dimension:

\begin{prp}\label{universal_minimal_domains}
Let $A\in\ts$ be universal. Then the following are equivalent:
\begin{enumerate}
\item $A$ is basic;
\item $A$ is a basic domain;
\item every $\alpha\in \End_\ts(A)$ is injective;
\item $A\prec B$ for every universal $B\in \ts$;
\item $A\thickapprox B$ for one (and therefore every) basic object $B\in \ts$;
\item no proper quotient of $A$ is universal;
\item no proper quotient of $A$ is a subalgebra of $A$.
\end{enumerate}
Any two basic objects are $\thickapprox$-equivalent domains of the same Krull-dimension $d_\fld(G)\le n=\mathrm{log}_p(|G|)$
with $d_\fld(G)>0$ if $G\ne 1$. With $\frak{B}$ we denote the $\thickapprox$-equivalence class of basic objects in $\ts$.
\end{prp}
\begin{proof}
Let $X\in\ts$ be a basic domain and $\alpha\in \End_\ts(X)$. Then $\alpha(X)\prec X$, hence
$X\prec \alpha(X)$, so $\Dim(X)=\Dim(\alpha(X))$ and $\alpha$ must be injective.\\
``(1) $\Rightarrow$ (2)":\  There is $\beta\in \ts(X,A)$ and $\gamma\in\ts(A,X)$,
so $\gamma\circ\beta\in\End_\ts(X)$ is injective, which implies that $\beta$ is injective
and therefore $X\prec A$. It follows that $A\prec X$, hence $A$ is a domain.\\
``(2) $\Rightarrow$ (3)":\  This has already been shown above.
(We didn't use the fact that $A$ is universal, there. So every minimal domain in $\ts$
satisfies (3)).\\
 ``(3) $\Rightarrow$ (4)":\  Since $A$ and $B$ are universal there exist morphisms
$\alpha\in\ts(A,B)$ and $\beta\in\ts(B,A)$ with $\beta\circ\alpha$ injective, because $A$
is minimal. Hence $A\prec B$.\\
``(4) $\Rightarrow$ (5)":\  This is clear.\\
``(5) $\Rightarrow$ (1)":\  $B\thickapprox A$ means that
$B\hookrightarrow A$ and $A\hookrightarrow B$. In that case $A$ is universal (minimal)
if and only if $B$ is universal (minimal). Choosing $B:=X$, it follows that $A$ is basic.\\
``(3)$\Rightarrow$ (6)":\  Now assume that every $\alpha\in \End_\ts(B)$ is injective and let
$B/I$ be universal for the $G$-stable ideal $I\unlhd B$. Then there is
$\gamma\in\ts(B/I,B)$ and the composition with the canonical map $c:\ B\to B/I$ gives
$\gamma\circ c\in\End_\ts(B)$. It follows that $I=0$.\\
``(6) $\Rightarrow$ (1)":\  Assume $B\prec A$. Then $B$ is universal and since $A$ is universal, there is
$\theta\in\ts(A,B)$ with $\theta(A)\le B$ universal. Hence $A\cong\theta(A)\prec B$ and
$A$ is basic. \\
Let $A,B\in\ts$ be basic, then $\ts(A,B)\ne\emptyset\ne \ts(B,A)$
implies that $A\prec B\prec A$, hence $A\thickapprox B$ and $\Dim(A)=\Dim(B)=:d_\fld(G)\le n=\mathrm{log}_p(|G|)$
(see Theorem \ref{arb_p_grp_sec0}).
Assume that $d_\fld(G)=0$. Then $X$ must be a Galois-field extension $K\ge\fld$ with Galois group $G$
and $K\hookrightarrow D_\fld$, which implies $K=\fld$ and $G=1$.\\
``(6) $\Rightarrow$ (7)":\  This is clear, because a quotient $A/I$ as subalgebra of $A$ would be universal.\\
``(7) $\Rightarrow$ (1)":\ We have $X\prec A$ and there is
$\theta\in\ts(A,X)$ with $\theta(A)\le X$ universal. It follows that $\theta(A)\prec A$, hence
$\ker\theta=0$ and $\theta(A)\cong A\thickapprox X$, so $A$ is basic.
\end{proof}

\begin{cor}\label{universal_minimal_domains_cor1}
Let $A\in\ts$ be a universal domain. Then $d_\fld(G)\le\Dim(A)$ and the following
are equivalent:
\begin{enumerate}
\item $A\in\frak{B}$;
\item $d_\fld(G)=\Dim(A)$;
\item If $C\in \ts$ with $C\prec A$, then $\Dim(C)=\Dim(A)$.
\end{enumerate}
\end{cor}
\begin{proof} The first statement and ``(1) $\Rightarrow$ (2)" follow immediately
from Proposition \ref{universal_minimal_domains}.\\
``(2) $\Rightarrow$ (3)":\  $C\prec A$ implies that $C$ is a universal domain and
$\Dim(C)\le\Dim(A)$. Hence $\Dim(A)=d_\fld(G)\le\Dim(C)\le \Dim(A)$.\\
``(3) $\Rightarrow$ (1)":\  Suppose $A$ is \emph{not} minimal. Then there
is $\alpha\in \End_\ts(A)$ with $\ker(\alpha)\ne 0$. Hence
$A/\ker(\alpha)\cong\alpha(A)=:C\prec A$. Clearly $\Dim(C)<\Dim(A)$.
\end{proof}

The results in Theorem \ref{standard_is_cycl_prj} can give useful bounds for $d_\fld(G)$.
Let $S={\rm Sym}(V^*)$ be as in Theorem \ref{tr_id_standard}; it has been shown in
\cite{universal} Theorem 2.7 that there exists $v^*\in V^{*G}$ such that $S/(v^*-1)S\in\ts$,
and so is $s$-projective and universal, if and only if ${\mathcal X}_V:=\langle V^g\ |\ 1\ne g\in G,\ g^p=1\rangle<V.$
In this case $d_\fld(G)\le {\rm dim}(V)-1$. If $\fld=\mathbb{F}_p$, the condition
${\mathcal X}_V<V$ implies ${\rm dim}(V)\ge n+1$ with $n=\mathrm{log}_p(|G|)$ (see \cite{universal} Proposition 3.3).
For certain $p$-groups (called ``CEA-groups" in \cite{universal}) the condition
${\mathcal X}_V<V$ is satisfied with ${\rm dim}(V)=n+1$, which then gives the known bound
$d_\fld(G)\le n$. For extension fields however, one can obtain sharp
bounds for $d_\fld$ and $e_\fld$ as the following examples show:

\begin{exs}\label{proj_d_bound_exs}
\begin{enumerate}
\item Let $q:=p^n$, $\mathbb{F}_q\le\fld$ and
$(C_p)^n\cong G$ a Sylow $p$-subgroup of ${\rm GL}_2(q)$ consisting
of upper triangular matrices. Set $V:=\fld^2=\fld e_1\oplus\fld e_2$ be the natural
$\fld G$-module, then ${\mathcal X}=\fld e_1< V$ and
${\rm Sym}(V^*)/(x_1-1)\cong B=\fld[Z]$ (see Theorem \ref{el_abel_ext_fld_d_is_1}),
proving again that $d_\fld(G)=1$.
\item Now let $\mathbb{F}_{q^2}\le\fld$ and let $G$ be a Sylow $p$-subgroup
of ${\rm SU}_3(q^2)$. Then $G$ can be represented as the group of matrices
$$g_{a,b}:=
\left(\begin{array}{ccc}
1&a&b\\
0&1&-a^q\\
0&0&1
\end{array}\right),\ a,b\in\mathbb{F}_{q^2},\ b+b^q+aa^q=0.$$
Let $V\cong\fld^3=\fld e_1\oplus \fld e_2\oplus \fld e_3$ be the natural
${\rm SU}_3(q^2)$-representation, then an elementary calculation shows that
${\mathcal X}_V=\langle e_1,e_2\rangle_\fld < V$, hence
${\rm Sym}(V^*)/(x_1-1)$ is $s$-projective and $d_\fld(G)\le 2$, so
$d_\fld(G)=e_\fld(G)=2$ by Corollary \ref{d_k_is two}. Note that
for $q=p$, $G$ is extraspecial (of exponent $p$, if $p\ge 3$).
\end{enumerate}
\end{exs}

\begin{cor}\label{ext_special}
Let $p\ge 3$, $\mathbb{F}_{p^2}\le\fld$ and $G$ be extraspecial of
order $p^3$ and of exponent $p$. Then $d_\fld(G)=e_\fld(G)=2$.
\end{cor}

Back again to basic objects; the $\thickapprox$-equivalence class $\frak{B}$ of basic objects contains cyclic domains:
\begin{cor}\label{universal_minimal_domains_cor3}
If $B\in\frak{B}$, then $B\thickapprox C$ with $C$ a cyclic domain.
\end{cor}
\begin{proof}
Let $\alpha\in\ts(D_\fld,B)$, then $C:=\alpha(D_\fld)\prec B$, hence $B\thickapprox C$ with cyclic
domain $C\in\ts$.
\end{proof}

In Lemma \ref{char_universal} we showed that every object $A\in\ts$ arises from
extending the quotient of a universal object by a ring with trivial $G$-action.
The class $\frak{B}$ consists of those objects from which all \emph{universal} objects arise by
extending invariants:

\begin{lemma}\label{char_basic}
An object $B\in\ts$ is basic, if and only if every universal object is
of the form $\frak W=\frak W^G\otimes_{B^G}B$ with embedding $B\hookrightarrow \frak W$.
\end{lemma}
\begin{proof}
If $B\in\ts$ has the described property and $X\in\ts$ is basic, then
$B\prec X$, so $X\thickapprox B$ and $B$ is basic.
Now assume that $B$ is basic and $\frak W$ is universal. Then by
Lemma \ref{char_universal}, $\frak W=\frak W^G\otimes_{S^G}S$ with $B/I\cong S\hookrightarrow \frak W$.
It follows that $B/I$ is universal, hence $I=0$ and $S\cong B$.
\end{proof}

We are therefore particularly interested in describing basic objects, i.e. minimal
subalgebras of $D_\fld$ which are also in $\ts$. However, with regard to minimality
the following has to be taken into account:
Since $D_\fld$ is the polynomial ring $\fld[x_g\ |\ 1\ne g\in G]$, we have $D_\fld^p=k^p[x_g^p\ |\ g\in G]$
and $\cap_{i\in \mathbb{N}} D_\fld^{p^i}=\cap_{i\in \mathbb{N}} k^{p^i}=k^{p^\infty},$
where $k^{p^\infty}$ denotes the maximal perfect subfield of $k$.
This implies that $C^p<C$ for every subring $C\ne k^{p^\infty} \le D_\fld$, hence there
will be in general no subalgebra of $D_\fld$ which is minimal with respect to
ordinary inclusion of $\fld$-subalgebras. (If $k$ is perfect and $C\le D_\fld$ is trace-surjective, then
$C^p<C$ is a proper inclusion of isomorphic objects in $\ts$).
The result in Corollary \ref{universal_minimal_domains_cor3} motivates the following

\begin{lemma}\label{alternative}
Let $C\in\ts$ be basic and cyclic. Then up to isomorphism $C\le D_\fld$ and
there exists $\chi\in \End_\ts(D_\fld)$ with $\chi(D_\fld)=C$ and
${\rm ker}(\chi_{|C})=0$. Moreover one of the following two situations can occur:
\begin{enumerate}
\item $C> \chi(C)> \cdots > \chi^n(C) > \chi^{n+1}(C) \cdots$ is an infinite descending chain of
properly contained, isomorphic $\fld-G$-subalgebras in $\ts$.
\item $C=\chi(C)$ and $D_\fld=C\oplus I$, where $I=\ker(\chi)\unlhd D_\fld$
is a $G$-stable ideal.
\end{enumerate}
\end{lemma}
\begin{proof} By Proposition \ref{universal_minimal_domains}, $C\prec D_\fld$, so
there is an embedding $\iota:\ C\hookrightarrow D_\fld$ and we
can assume $C=\iota(C)=\fld[W]\le D_\fld$ with $w\in W$ of trace $1$ and
$W\cong \fld G$ as $\fld G$-module. Then $W=\langle wg\ |\ g\in G\rangle$ and the map
$x_g\mapsto wg$ defines a $G$-equivariant $k$-algebra epimorphism $\theta:\ D_\fld\to C$.
Set $\phi:=\theta\circ\iota$ and $\chi:=\iota\circ\theta$,
then $\phi\in\End_\ts(C)$ is injective with image $C\cong\phi(C)=\theta(C)\le C$, so
$\chi(C)=\iota\circ\theta(\iota(C))=\phi(C)\le C.$
Suppose $\chi^{n+1}(C)=\chi^n(C)$ and let $c\in C$; then $\chi^n(c)=\chi^{n+1}(c')$ for some $c'\in C$,
so $c-\chi(c')\in \ker(\chi^n_{|C})\subseteq\ker(\phi^{n+1})=0$. Hence $c=\chi(c')$,
$\chi(C)=C=\phi(C)$ and
$\phi=\theta\circ\iota$ is an automorphism of $C$. We conclude $D_\fld=C\oplus \ker(\chi)$.
\end{proof}

If $k=k^p$, we have already seen that case (1) actually occurs. For general $k$,
the homomorphism $\tilde F=\fld[X_g\ |\ g\in G] \to \tilde F$ defined by $X_g\mapsto X_g^p$ induces
a {\bf Frobenius-endomorphism}
$$\Phi:\ D_\fld=\fld[x_g\ |g\in G] \to D_\fld,\ \alpha(x_1,\cdots,x_g,\cdots)\mapsto \alpha(x_1^p,\cdots,x_g^p,\cdots),$$
which in the case $k=\mathbb{F}_p$ coincides with the ordinary power map $a\mapsto a^p$.
It follows that $D_\fld>\Phi(D_\fld)>\cdots \Phi^n(D_\fld)>\Phi^{n+1}(D_\fld)>\cdots$. Similarly, for
every subalgebra $C_0=\mathbb{F}_p[V]\le D_{\mathbb{F}_p}$ with subspace
$1\in V\cong \mathbb{F}_pG$ we have
$$C_0>\Phi(C_0)=C_0^p>\cdots >C_0^{p^n}>C_0^{p^{n+1}}>\cdots$$
and therefore the subalgebra $C:=k\otimes_{\mathbb{F}_p}C_0 \le D_\fld$ satisfies
$$C>\Phi(C)>\cdots >\Phi^n(C)>\Phi^{n+1}(C)>\cdots.$$
In the rest of this section, and in fact the paper, we will study the second case of lemma \ref{alternative},
which also occurs naturally and, in many respects, is the more interesting situation.

If $S\le D_\fld$ is standard, then there is a projection morphism
$\chi:\ D_\fld\to S\hookrightarrow D_\fld$, which is an idempotent in $\End_\ts(D_\fld)$.
The following has been shown in \cite{nonlin}:

\begin{lemma}\label{endo_krit_standard}[\cite{nonlin} Lemma 5.1]
Let $S\hookrightarrow D_\fld$ be a trace-surjective $G$-algebra, then the following are equi\-valent:
\begin{enumerate}
\item $S$ is standard.
\item $\exists\ \chi=\chi^2\in ({\rm End}_\ts(D_\fld)$ with $S=\chi(D_\fld)$.
\item $\exists\chi\in ({\rm End}_{k-{\rm alg}}(D_\fld))^G$ with
$\chi^2(x_1)=\chi(x_1)=:w\in S=\fld[wg\ |g\in G]$.
\item $\exists\ w=W(x_1,x_{g_2},\cdots,x_{g_{|G|}})\in S$ with ${\rm tr}(w)=1$,
$w=W(w,wg_2,\cdots,w_{g_{|G|}})$ and $S=\fld[wg\ |g\in G]\le D_\fld.$
\end{enumerate}
\end{lemma}

Let $S\le D_\fld$ be standard. Since $D_\fld$ is a polynomial $\fld$-algebra
it follows from \cite{costa} Corollary 1.11, that $S$ is a regular UFD.

\begin{df}\label{reflexive_points} (see \cite{nonlin}[Definition 4])
A point $w\in D_\fld$ will be called {\bf reflexive}, if
$$w=W(x_1,\cdots,x_g\cdots)=W(w,\cdots,wg,\cdots)=\theta(w),$$
where $\theta\in ({\rm End}_{k-{\rm alg}}(D_\fld))^G$ is defined by
$x_g\mapsto w\cdot g$ $\forall g\in G$.
\end{df}

By definition a trace-surjective  $G$-algebra is cyclic, if and only if it is generated as an algebra by the $G$-orbit of one point. Lemma \ref{endo_krit_standard} shows, that
the standard subalgebras of $D_\fld$ are precisely the subalgebras generated by the $G$-orbit of a \emph{reflexive} point.
\\[1mm]
Let $G_1, G_2$ be two finite $p$-groups and $A_i\in\ts_{G_i}$ with point $a_i\in A_i$
for $i=1,2$. Then $a_1\otimes_\fld a_2$ is easily seen to be a point of
$A_1\otimes_\fld A_2\in\ts_{G_1\times G_2}$. Moreover, $D_\fld(G_1)\otimes_\fld D_\fld(G_2)$
is standard universal, i.e. a retract of $D_\fld(G_1\times G_2)$ (see \cite{nonlin} Section 5, Example 3).
If the $A_i$'s are universal with $\theta_i\in\ts_{G_i}(A_i,D_\fld(G_i))$,
then $\theta_1\otimes\theta_2\in \ts_{G_1\times G_2}(A_1\otimes_\fld A_2,D_\fld(G_1)\otimes_\fld D_\fld(G_2)),$
hence $A_1\otimes_\fld A_2$ is universal in $\ts_{G_1\times G_2}$.
Clearly the polynomial ring $\tilde D:=D_\fld(G_1)\otimes_\fld D_\fld(G_2)$ can be viewed as an
object in $\ts_{G_1}$ or $\ts_{G_2}$ by restricting the action accordingly.
In that way $\tilde D_{|G_1}\cong D_\fld(G_1)\otimes_\fld \fld[T_1,\cdots,T_{|G_2|-1}]$ is a polynomial ring
over $D_\fld(G_1)$ with trivial $G_1$-action on $\fld[T_1,\cdots,T_{|G_2|-1}]$.
Let $R\in\ts_{G_1}$ and $\phi\in\ts_{G_1}(D_\fld(G_1),R)$, then any map $T_j\mapsto r_j\in R^{G_1}$
extends $\phi$ to a morphism $\tilde\phi\in\ts_{G_1}(\tilde D_{|G_1},R)$, which shows that
$\tilde D_{|G_i}$ is universal in $\ts_{G_i}$.\\
Suppose that $\phi\in \ts_{G_1\times G_2}(\tilde A,\tilde D)$
for $\tilde A:=A_1\otimes_\fld A_2$ with $i_1\in\ts_{G_1}(A_1,\tilde A)$ the
canonical morphism. Then the composition $\phi_{|G_1}\circ i_1$ is in $\ts_{G_1}(A_1,\tilde D_{|G_1})$,
hence $A_1$ is universal. We summarize:

\begin{prp}\label{direct_product_group}
Let $G_1$ and $G_2$ be two finite $p$-groups with $A_i\in\ts_{G_i}$ for $i=1,2$. Then
$A_1\otimes_\fld A_2\in\ts_{G_1\times G_2}$ and the following hold:
\begin{enumerate}
\item $A_i$ universal in $\ts_{G_i}$ for $i=1,2$ $\iff$  $A_1\otimes_\fld A_2$ is universal in
$\ts_{G_1\times G_2}$;
\item $A_i$ standard universal in $\ts_{G_i}$ for $i=1,2$ $\Rightarrow$  $A_1\otimes_\fld A_2$ is standard
universal in $\ts_{G_1\times G_2}$;
\item $d_\fld(G_1\times G_2)\le d_\fld(G_1)+d_\fld(G_2)$.
\end{enumerate}
\end{prp}

We close this section by illustrating the above notions in the case of
elementary-abelian $p$-groups. We need some notation and a lemma:
Define $\partial_n(T)\in \fld[X_1,\cdots,X_{n-1}][T]$ to be the following $n\times n$-determinant:
$$\partial_n(T)
=\partial_n(\underline X,T):=\left| \begin{array}{cccc}
X_1&\cdots &X_{n-1}&T\\
X_1^p&\cdots &X_{n-1}^p&T^p\\
\cdots&\cdots&\cdots&\cdots\\
X_1^{p^{n-1}}&\cdots &X_{n-1}^{p^{n-1}}&T^{p^{n-1}}
\end{array} \right|,$$
and set $F_{n-1}(T):= \prod_{x\in V}(T-x),$
where $V:=\langle X_1,\cdots,X_{n-1}\rangle_{\mathbb{F}_p}$.

\begin{lemma}\label{Dickson_invariants}
The following hold:
\begin{enumerate}
\item $\partial_n(T)=\partial_{n-1}(X_{n-1})\cdot F_{n-1}(T)$;
\item for any $\alpha_1,\cdots,\alpha_n\in\fld$,
$\partial_n(\alpha_1,\cdots,\alpha_{n-1},\alpha_n)\ne 0$ if and only if
the set $\{\alpha_1,\cdots,\alpha_n\}$ is linearly independent over $\mathbb{F}_p$.
\end{enumerate}
\end{lemma}
\begin{proof}
(1):\ For every $x\in V$ we have $\partial_n(x)=0$, so considering the $T$-degree we obtain
$\partial_n(T)=c\cdot F_{n-1}(T)$ with $c$ being the coefficient of
$\partial_n(T)$ at $T^{p^n}$, hence $c=\partial_{n-1}(X_{n-1})$.\\
(2):\ Assume that $\{\alpha_1,\cdots,\alpha_n\}\subseteq\fld$ is linearly independent over $\mathbb{F}_p$ and set $$f(T):=\prod_{x\in W}(T-x)\ {\rm with}\  W:=\langle \alpha_1,\cdots,\alpha_{n-1}\rangle_{\mathbb{F}_p}.$$
Then we have
$\partial_n(\alpha_1,\cdots,\alpha_{n-1},\alpha_n)=$
$\partial_{n-1}(\alpha_1,\cdots,\alpha_{n-2},\alpha_{n-1})\cdot f(\alpha_n).$
By induction the first factor is nonzero and $f(\alpha_n)\ne 0$,
since $\alpha_n\not\in W$; hence $\partial_n(\alpha_1,\cdots,\alpha_{n-1},\alpha_n)\ne 0$.
Conversely, if $\partial_n(\alpha_1,\cdots,\alpha_{n-1},\alpha_n)\ne 0$, then, again by induction,
$\{\alpha_1,\cdots,\alpha_{n-1}\}$ is linearly independent over $\mathbb{F}_p$. Moreover
$f(\alpha_n)\ne 0$, so $\alpha_n\not\in W$ and $\{\alpha_1,\cdots,\alpha_n\}$ is linearly independent over $\mathbb{F}_p$.
\end{proof}

\noindent
Let $G$ be an elementary-abelian group of order $p^n$. We
identify $G$ with the additive group $(\mathbb{F}_p^n,+)$ and
write an element $g\in G$ as a vector $g=\sum_{i=1}^n g_i e_i$ with $g_i\in\mathbb{F}_p$ and
$e_i$ the standard basis vector of $\mathbb{F}_p^n$.
Set $\mho:=\fld[Y_1,\cdots,Y_n]$, the polynomial ring in $n$ variables,
then $G$ acts on $\mho$ by the rule $(Y_i)g=Y_i-g_i$ for all $i$, hence
$\mho=U_1\otimes_\fld U_2\otimes_\fld \cdots \otimes_\fld U_n$, with $U_i=\fld[Y_i]\in\ts_{G_i}$
and $G_i:=\langle e_i\rangle$.
It follows from \cite{nonlin} Proposition 3.2 that every
$U_i\in\ts_{G_i}$ is a basic and standard subalgebra of $D_\fld(G_i)$, hence by
Proposition \ref{direct_product_group}, $\mho$ is a standard universal subalgebra of $D_\fld(G)$.
\\[1mm]
Now assume that $\fld$ contains an $n$-dimensional $\mathbb{F}_p$-subspace
$W:=\langle\alpha_1,\cdots,\alpha_n\rangle$; then there is an embedding of abelian
groups
$\alpha:\ G\to W\le\fld^+,\ g\mapsto \alpha_g:=\sum_{i=1}^n g_i\alpha_i$.
Consider a univariate polynomial ring $\fld[Z]$ with (nonlinear) $G$-action extending the maps
$Z\mapsto (Z)g=Z-\alpha_g$ to $\fld$-algebra automorphisms.
The corresponding $\fld-G$-algebra will be denoted by $B_\alpha$. Then the map
$Z\mapsto \sum_{i=1}^n\alpha_iY_i$ extends to a $G$-equivariant
morphism of $\fld-G$-algebras $\theta:\ B_\alpha\to \mho$. It follows from
Lemma \ref{Dickson_invariants} that there exists a matrix
$(f_{ij})^T:=(\alpha_i^{p^{j-1}})^{-1}\in{\rm GL}_n(\fld)$, i.e.
such that $\sum_{j=0}^{n-1}f_{ij}\alpha_k^{p^j}=\delta_{ik}$.
Set $f_i(Z):=\sum_{j=0}^{n-1} f_{ij}Z^{p^j}\in\fld[Z]$, then
$f_i(\alpha_g)=g_i$ for every $g\in G$.
Now define a $\fld$-algebra morphism $\psi:\ \mho\to B_\alpha$ by extending the map
$Y_i\mapsto f_i(Z)$. Then $f_i(\mu+\lambda)=$ $f_i(\mu)+f_i(\lambda)$ for $\mu,\lambda\in\fld$ and
$f_i(\lambda)=\lambda_i$, whenever $\lambda=\sum_{i=1}^n\lambda_i\alpha_i$ with $\lambda_i\in\mathbb{F}_p$.
Hence $\psi\circ\theta(Z)=h(Z)\in\fld[Z]$ is a polynomial of degree less than $p^n$ such that
$h(\lambda)-\lambda=\sum_{i=1}^n\alpha_if_i(\lambda)-\lambda=$
$\sum_{i=1}^n\alpha_i\lambda_i-\lambda=0$ for all $\lambda\in W$. It follows that $h(Z)=Z$.
Moreover, for every $g\in G$ we have $\psi((Y_i)g)=$ $\psi(Y_i-g_i)=$
$\psi(Y_i)-g_i=$ $f_i(Z)-f_i(\alpha_g)=$
$f_i(Z-\alpha_g)=$ $f_i((Z)g)=$ $(f_i(Z))g=$ $(\psi(Y_i))g.$
This shows that $\psi$ is a $G$-equivariant retraction. In particular $B_\alpha$ is a trace-surjective
retract of $\mho$, hence a standard and basic universal algebra in $\ts_G$.
\\
Let $\beta:\ G\to\fld^+,\ g\mapsto \beta_g:=\sum_{i=1}^n g_i\beta_i$
be a different embedding of abelian groups and define $B_\beta$ to be
the univariate polynomial ring $\fld[Z]$ with $G$-action given by $Z\mapsto (Z)g=Z-\beta_g$.
Since the set $\{\beta_1,\cdots,\beta_n\}$ is linearly independent over $\mathbb{F}_p$,
there are $(\lambda_0,\cdots,\lambda_{n-1})\in \fld$ with
$\sum_{j=0}^{n-1}\lambda_j\cdot\beta_i^{p^j}=\alpha_i$ for $i=1,\cdots,n$.
Let $L_{\alpha,\beta}:\ B_\alpha\to B_\beta$ be the algebra homomorphism extending
the map $Z\mapsto f_{\alpha,\beta}(Z):=\sum_{j=0}^{n-1}\lambda_j\cdot Z^{p^j}$.
Then $L_{\alpha,\beta}$ is injective, because $f_{\alpha,\beta}(Z)\not\in\fld$ and
$L_{\alpha,\beta}((Z)g)=$ $L_{\alpha,\beta}(Z-\alpha_g)=$ $L_{\alpha,\beta}(Z)-\alpha_g=$
$f_{\alpha,\beta}(Z)-f_{\alpha,\beta}(\beta_g)=$ $f_{\alpha,\beta}(Z-\beta_g)=$
$(L_{\alpha,\beta}(Z))g.$
So $L_{\alpha,\beta}$ is a $G$-equivariant embedding $B_\alpha\hookrightarrow B_\beta$. In a similar way we see that $L_{\beta,\alpha}\in \ts_G(B_\beta,B_\alpha)$
is injective, hence $B_\beta$ is universal and indeed $B_\alpha\thickapprox B_\beta$.

We summarize

\begin{thm}\label{el_abel_ext_fld_d_is_1}
Let $G$ be elementary-abelian of order $p^n$ and $\mho:=\fld[Y_1,\cdots,Y_n]\in\ts_G$ as described above.
Then the polynomial ring $\mho$ is a standard universal subalgebra of $D_\fld(G)$. \\
Assume now that $\dim_{\mathbb{F}_p}(\fld)\ge n$, then there is an embedding of abelian
groups $$\alpha:\ G\to W:=\langle\alpha_1,\cdots,\alpha_n\rangle_{\mathbb{F}_p}\le\fld^+,\ g\mapsto \alpha_g:=\sum_{i=1}^n g_i\alpha_i$$ and the following hold:
\begin{enumerate}
\item The univariate polynomial ring $B_\alpha=\fld[Z]_\alpha$ with $G$-action as described above is a retract of $\mho$ in $\ts_G$ and a standard universal and basic object in $\ts_G$. In particular $d_\fld(G)=1$.
\item Every basic object in $\ts_G$ which is also a normal ring is of the form $B_\beta$ for some
embedding of abelian groups $\beta:\ G\hookrightarrow\fld^+$.
\item Two normal basic algebras $B_\alpha, B_\beta\in\ts_G$
are isomorphic if and only if
$\alpha= c\cdot\beta$ for some $0\ne c\in\fld$. They are conjugate under an outer automorphism
of $G$ if and only if $\alpha(G)=c\cdot\beta(G)$ for some $0\ne c\in\fld$.
\item Let $\underline\alpha:=\alpha^{(1)},\cdots,\alpha^{(n)}$ be $n$ not necessarily distinct embeddings $G\hookrightarrow\fld^+$. Then
    $$\mho\cong B_{\underline\alpha}^{\otimes n}\cong B_{\alpha^{(1)}}\otimes_\fld
    B_{\alpha^{(2)}}\otimes_\fld\cdots\otimes_\fld B_{\alpha^{(n)}}.$$
\end{enumerate}
\end{thm}
\begin{proof}
(1):\ This has already been shown. (2):\ Let $N\in\ts$ be basic and normal. Then $N\hookrightarrow \mho$
and it follows from \cite{eakin} that $N\cong\fld[T]$ is a univariate polynomial algebra. It is clear that the $G$-action is of the form $(T)g=T-\beta(g)$ with $\beta\in {\rm Hom}(G,\fld^+)$. Since $\ker(\beta)\le G$ acts
trivially on $N$, which is a faithful $\fld G$-module, we must have $\ker(\beta)=1$, so $\beta$ is injective
and $N\cong B_\beta\in\ts$.\ (3):\ Let $\eta\in {\rm Aut}(G)\cong{\rm GL}_n(\mathbb{F}_p)$ and assume that
$\theta$ is an $\fld$-algebra isomorphism $B_\alpha\to (B_\beta)^\eta$.
Then $\theta(Z)=c\cdot Z+\mu$ with $c,\mu\in\fld$ and $c\ne 0$, such that
$\theta((Z)g)=$ $\theta(Z-\alpha(g))=$ $c Z+\mu-\alpha(g)=$
$(\theta(Z))\eta(g)=$ $(c Z+\mu)\eta(g)=$ $c (Z-\beta(\eta(g)))+\mu.$
This implies $\alpha(g)=c\cdot\beta(\eta(g))$ for all $g\in G$ and the last statement
in (3) follows easily.\\
(4):\
As above we define $\theta_s\in\ts_G(B_{\alpha^{(s)}},\mho)$ by
$Z\mapsto \sum_{j=1}^n\alpha^{(s)}_jY_j$.
Assume first that $\alpha^{(s)}=(\alpha^{(1)})^{(p^{s-1})}$ with
$\alpha^{(s)}_i=(\alpha_i^{(1)})^{p^{s-1}}$.
Set $\Gamma:=(\gamma_{ij})=(\alpha_i^{p^{j-1}})^{-1}\in {\rm GL}_n(\fld)$,
then $Y_k=$ $\sum_{j=1}^n \gamma_{kj}\theta_j(Z)$, hence
the coproduct morphism $\Theta:=\coprod_{s=1}^n{\theta_s}:=\theta_1\otimes\cdots\otimes\theta_n$
is a surjective $G$-equivariant algebra homomorphism from $\otimes_{i=1}^nB_{\alpha^{(i)}}$ to $\mho$.
Since both algebras are polynomial algebras of Krull-dimension $n$, $\Theta$ is
an isomorphism. Clearly $\mho$ and each of the
$B_{\alpha^{(i)}}$ are triangular and therefore erasable.
It follows that $\mho\cong B_{\alpha^{(1)}}\otimes_\fld\fld[\lambda_2,\cdots,\lambda_{n}]$
with $\fld^{[n-1]}\cong\fld[\lambda_2,\cdots,\lambda_{n}]\le\mho^G$.
Now we take $\alpha^{(i)}$ for $i=2,...,n$ to be arbitrary embeddings
$G\hookrightarrow\fld^+$. As before we see that $B_{\underline\alpha}^{\otimes n}\cong$
$B_{\alpha^{(1)}}\otimes_\fld\fld[\mu_2,\cdots,\mu_{n}]$
with $\fld^{[n-1]}\cong\fld[\mu_2,\cdots,\mu_{n}]\le(B_{\underline\alpha}^{\otimes n})^G$, so
$\mho\cong B_{\underline\alpha}^{\otimes n}$. This finishes the proof.
\end{proof}

\begin{cor}\label{el_abel_ext_fld_d_is_1_cor}
Let $G\cong \mathbb{F}_{p^n}^+$ and $\mathbb{F}_{p^s}\le\fld$ for some
$s\le n$. Then
$$d_\fld(G)\le\begin{cases} n/s&\text{if $s$ divides $n$}\cr
\lfloor{n/s}\rfloor+1&\text{otherwise}\cr
\end{cases}$$
where $\lfloor x\rfloor$ is the largest integer $\le x$.
\end{cor}
\begin{proof}
Let $n=ms+r$ with $0\le r<s$. Then Proposition \ref{direct_product_group}
and Theorem \ref{el_abel_ext_fld_d_is_1} give
$d_\fld((C_p)^n)\le $
$m\cdot d_\fld((C_p)^s)+d_\fld((C_p)^r),$ which is equal to
$m=n/s$, if $r=0$ and equal to $m+1=\lfloor{n/s}\rfloor+1$ otherwise.
\end{proof}

With the help of Theorem \ref{erasbility_theorem}
we can classify the minimal normal generators and minimal normal $s$-projective objects of $\ts$ in the case where
$G$ is elementary-abelian and $\fld$ is large enough. We will use the notation
introduced before Theorem \ref{el_abel_ext_fld_d_is_1}:

\begin{prp}\label{el_ab_min_gen}
Let $G$ be elementary-abelian of order $p^n$ and $\dim_{\mathbb{F}_p}(\fld)\ge n$ and let
$\Gamma\in\ts$ be a normal ring. Then the following are equivalent:
\begin{enumerate}
\item $\Gamma$ is a generator and minimal in $\ts$;
\item $\Gamma\cong B_\alpha=\fld[Z]_\alpha$ for some embedding $\alpha:\ G\hookrightarrow k^+$;
\item $\Gamma$ is an $s$-projective and minimal object in $\ts$.
\end{enumerate}
\end{prp}
\begin{proof}
``(1) or (3) $\Rightarrow$ (2) ":\ Since every generator and every $s$-projective object is universal,
this follows from Theorem \ref{el_abel_ext_fld_d_is_1}.\ ``(2) $\Rightarrow$ (3) ":\
This also follows directly from Theorem \ref{el_abel_ext_fld_d_is_1}.
``(2) $\Rightarrow$ (1) ":\ Since $B_\alpha$ is basic, it is minimal in $\ts$, so
it remains to show that $B_\alpha$ is a generator. But $B_\alpha\in\ts$ is triangular and therefore
erasable, so it follows from Theorem \ref{erasbility_theorem}
that $B_\alpha$ is an s-generator, hence a generator (see Lemma \ref{generators in ts}).
\end{proof}

\section{Basic Algebras and the Essential Dimension of $G$}\label{sec_Basic Algebras}

In this section we are going to point out interesting connections to the notion of ``essential dimension"
of a group, as defined by Buhler and Reichstein (\cite{Buhler:Reichstein}).
Let for the moment $\fld$ be an arbitrary field and $\frak{G}$ an arbitrary finite group, acting
faithfully on the finite-dimensional $\fld$-vector space $V$.
Then the essential dimension $e_\fld(\frak{G})$ is defined to be the minimal transcendence degree
over $\fld$ of a field $E$ with $\fld\le  E\le \fld(V^*):={\rm Quot}(S_\fld(V^*))$ such that $\frak{G}$ acts faithfully on $ E$. It can be shown, that the value $e_\fld(\frak{G})$ only depends
on the group $\frak{G}$ and the field $\fld$, but not on the choice of the faithful representation
(see \cite{Buhler:Reichstein} Theorem 3.1., if $\fld$ has characteristic $0$ and
\cite{berhuy_favi} Proposition 7.1 or \cite{Kang_essdim} for arbitrary field $\fld$).
For an arbitrary field $ K\ge\fld$ together with an embedding of $\frak{G}$ in $\Aut_\fld(K)$, define
$$e_\fld( K):=\min\{{\rm tr.deg}_\fld E\ |\ \fld\le E\le K,\ E\ \text{is $\frak{G}$-stable
with faithful action}\},$$
in other words, $e_\fld(K)$ is the minimum transcendence degree of a Galois field extension
$E/E^\frak{G}$ containing $\fld$ and contained in $K$.

\begin{lemma} \label{ess_dim_def}
$e_\fld(\frak{G})=
\max_{\fld\le K\atop\frak{G}\le\Aut_\fld( K)} e_\fld( K)=
\max_{\fld\le K\atop\frak{G}\le\Aut_\fld( K)}
(\min_{\fld\le E\le K\atop\frak{G}\le\Aut_\fld( E)}{\rm tr.deg}_\fld E).$
\\
Moreover, $e_\fld(\frak{G})=e_\fld(L)$ for any field $L\le\fld(V^*)$ with  $\frak{G}\le\Aut_\fld(L)$.
\end{lemma}
\begin{proof}
Define $\tilde e_\fld(\frak{G}):=\max_{\fld\le K\atop\frak{G}\le\Aut_\fld( K)} e_\fld(K)$.
By definition $e_\fld(\frak{G})=e_\fld(k(V^*))\le$
$\tilde e_\fld(\frak{G})$. By \cite{Kang_essdim} Proposition 2.9,
$e_\fld(K)\le e_\fld(\frak{G})$ for any field $K$ with
$\frak{G}\le\Aut_\fld(K)$, hence $\tilde e_\fld(\frak{G})\le e_\fld(\frak{G}).$
Now pick any field $L\le\fld(V^*)$ with faithful $\frak{G}$-action.
Then by the definitions we have
$e_\fld(\frak{G})\le e_\fld(L)\le\tilde e_\fld(\frak{G}),$
which finishes the proof.
\end{proof}

\begin{df}\label{essential_G_field}
A field extension $L\ge \fld$ with $\frak{G}\le \Aut_\fld(L)$ will be
called a {\bf $\frak{G}$-field } (over $\fld$). If ${\rm tr.deg}_\fld L=e_\fld(\frak{G})=e_\fld(L)$,
then $L$ will be called an {\bf essential} $\frak{G}$-field (over $\fld$).
\end{df}

Now let $\fld$ again be of characteristic $p>0$, let $G$ be a $p$-group and choose
$V:=V_{reg}$. Let $B\in\ts$ be basic (and cyclic, if we wish), then
$B\prec D_\fld$ with
$${\rm Quot}(B)\le {\rm Quot}(D_\fld)\le \fld(V^*).$$
Clearly $G$ acts faithfully on ${\rm Quot}(B)$, so $d_\fld(G)=\Dim(B)\ge e_\fld(G)$.
On the other hand, let $\fld\le K$ be essential with $K\le \fld(V^*)$, then we can
choose a point $a\in K$ and consider the algebra
$A:=\fld[a^G]:=\fld[(a)g\ |\ g\in G]\in \ts$.
It follows from the definition of $e_\fld(G)$ that $\Dim(A)=e_\fld(G)$. Moreover,
the map $(x_1)g\mapsto (a)g$ extends to a surjective morphism $\phi:\ D_\fld\to A$, so
$A\cong D_\fld/\ker(\phi)$ is a cyclic domain in $\ts$. If $\frak{U}\le D_\fld$ is universal,
there is also a morphism $\alpha\in\ts(\frak{U},A)$ and since $\alpha(\frak{U})\subseteq K$ with faithful $G$-action on $\alpha(\frak{U})$ it follows again from the definition of $e_\fld(G)$ that $\Dim(A)=\Dim(\alpha(\frak{U}))=e_\fld(G)$.
Hence
$d_\fld(G)=\Dim(B)\ge \Dim(\alpha(\frak{U}))=$ ${\rm tr.deg}_\fld({\rm Quot}(\alpha(\frak{U})))=e_\fld(G)$,
so $K\ge {\rm Quot}(\alpha(\frak{U}))$ is an algebraic extension.
Note that $\alpha(\frak{U})\cong \frak{U}/{\rm p}$ for some $G$-stable prime ideal ${\rm p}\unlhd \frak{U}$.
Conversely, if $\wp\in {\rm Spec}(\frak{U})$ is $G$-stable such that
$k(\wp):={\rm Quot}(\frak{U}/\wp)\le K$, then $K$ is algebraic over $k(\wp)$, so
$\fld\le k(\wp)$ is essential. It follows that $e_\fld(G)$ is the minimum of the transcendence degrees
of ``embedded residue class fields" ${\rm tr.deg}_\fld k(\wp)$ of those $G$-stable prime ideals $\wp\unlhd\frak{U}$ that satisfy $k(\wp)\hookrightarrow{\rm Quot}(\frak{U}).$
This motivates the following

\begin{df}\label{S_A}
Let $A\in\ts$ with total ring of quotients $Q(A):={\rm Quot}(A)$. With
${\rm Spec}(A)^G$ we denote the set of $G$-stable prime ideals of $A$.
We also define
$$\calS_{A}:=\{k(\wp)\ |\ \wp\in {\rm Spec}(A)^G,\ \exists\ \text{a $G$-equivariant embedding}
\  k(\wp)\hookrightarrow Q(A)\},$$
the set of all ``embedded residue class fields" of $G$-stable prime ideals of $A$.
\end{df}
Note that if $A\in\ts$ is a domain, then $Q(A)=k(0)\in\calS_A$.
We can now summarize

\begin{prp}\label{ess_dim_squeeze}
Let $\fld$ be a field of characteristic $p>0$, $G$ a group of order $p^n$
and $\frak{U}\le D_\fld$ a universal trace-surjective algebra (e.g. any basic
algebra). Set $d_{{\rm dom},\fld}(G):={\rm min}\{\Dim(C)\ |\ C\in\ts\ |\ C\ (cyclic)\ {\rm domain}\}$, then
$$n\ge d_\fld(G)\ge e_\fld(G)\ge d_{{\rm dom},\fld}(G).$$
Moreover $e_\fld(G)=e_\fld(Q(\frak{U}))=$
$\min\{{\rm tr.deg}_\fld k(\wp)\ |\ k(\wp)\in \calS_\frak{U}\}$ and every essential $G$-field $K\ge \fld$
is algebraic over an essential $G$-field of the form  $\fld\le k(\wp)\in \calS_\frak{U}$.
\end{prp}

\noindent
Note that we can choose $\frak{U}$ to be, for example, the polynomial algebra
$U=\fld[Y_1,\cdots,Y_n]$ mentioned in Theorem \ref{arb_p_grp_sec0}.
So, at the expense of replacing a faithful linear action of $G$ on
$S(V^*)$ by a nonlinear action on $U$, one can reduce the dimensions
of rings from which to construct essential $G$-fields.
If for example $G$ is cyclic of order $p^n$, the smallest faithful representation
has dimension $p^{n-1}+1$, whereas $U$ has Krull-dimension $n$.
Since every \emph{basic} algebra $B\in\frak{B}$ is embedded into $D_\fld$,
we have the following ``intrinsic description" of the essential dimension:

\begin{cor}\label{ess_dim_intrinsic}
Let $B$ be any basic algebra in $\ts$, then
$$e_\fld(G)=e_\fld(Q(B))=
\min\{{\rm tr.deg}_\fld k(\wp)\ |\ k(\wp)\in\calS_B\}.$$
\end{cor}

In Proposition \ref{universal_minimal_domains} (7) we proved that a universal algebra
$A\in\ts$ is basic if and only if it does not have any
``embedded" trace-surjective proper factor rings. The following
is a criterion in a similar spirit for the situation where $d_\fld(G)=e_\fld(G)$:

\begin{lemma}\label{ess_dim_eq_dk_crit}
For any universal domain $A\in\ts$ the following are equivalent:
\begin{enumerate}
\item $\calS_{A}=\{Q(A)\}$;
\item $A$ is basic and $d_\fld(G)=e_\fld(G)$;
\item $e_\fld(G)=\Dim(A)$.
\end{enumerate}
If these hold, $Q(A)$ is an essential $G$-field and all the others are
algebraic extensions thereof.
\end{lemma}
\begin{proof}
(1) $\Rightarrow$ (2):\ It follows from Proposition
\ref{universal_minimal_domains} (7) that $A$ is basic, hence by
Corollary \ref{ess_dim_intrinsic},
$e_\fld(G)={\rm tr.deg}_\fld k(\wp)$ for some
$k(\wp)\in\calS_{A}$. So $k(\wp)=Q(A)$, $\wp=0$ and $d_\fld(G)=\Dim(A)=e_\fld(G)$.\\
(2) $\Rightarrow$ (3):\ This is obvious, since $\Dim(A)=d_\fld(G)$.\\
(3) $\Rightarrow$ (1):\ Since $A$ is universal, Corollary \ref{universal_minimal_domains_cor1}
yields $\Dim(A)=e_\fld(G)\le d_\fld(G)\le \Dim(A)$, so $A$ is basic.
Now assume $k(\wp)\in \calS_A$; then, $k(\wp)\le Q(A)$ and by Corollary \ref{ess_dim_intrinsic},
$e_\fld(G)=e_\fld(Q(A))\le {\rm tr.deg}_\fld k(\wp)=\Dim(A/\wp)\le \Dim(A)=e_\fld(G),$
so $\wp=0$.
\end{proof}

Let $T:=\fld(x_1,\cdots,x_n)$ be a purely transcendental field extension and
$L\le T$ a subfield of transcendence degree $m\le n-1$. Then it follows
from a result of Roquette and Ohm (see Proposition 8.8.1. \cite{Jensen_Gen_pols}) that
$L$ can be embedded into $\fld(x_1,\cdots,x_{n-1})$. An obvious induction shows that,
indeed, $L$ can be embedded into $\fld(x_1,\cdots,x_m).$
This can be used to obtain the following result:

\begin{prp}\label{ess_dim_lower bound}
Let $A\in\ts$ with $A\le D_\fld$ and assume that $G$ is not isomorphic to a subgroup of
$\Aut_\fld(L)$ for any intermediate field $\fld< L\le\fld(x_1,\cdots,x_{m-1})$ with $L=k(\wp)\in\calS_A$.
Then $m\le e_\fld(G)$.
\end{prp}
\begin{proof} By Proposition \ref{ess_dim_squeeze} there is an essential $G$-field
$k(\wp)\in\calS_A$ with $k(\wp)\le Q(A)\le Q(D_\fld)\cong \fld(x_1,\cdots,x_{|G|-1})$.
Assume $e_\fld(G)<m$, then $k(\wp)$ can be embedded into
$\fld(x_1,\cdots,x_{m-1})$ and $G\le \Aut(k(\wp))$. This contradiction finishes the proof.
\end{proof}

By L\"uroth's theorem, any intermediate field $\fld<L\le\fld(x_1)$ is
rational, i.e. isomorphic to $\fld(x_1)$ and therefore
$\Aut_\fld(L)\cong {\rm PGL}_2(\fld)$. From this we obtain:

\begin{prp}\label{ess_dim_1}
Let $\fld$ be a field of characteristic $p>0$ and $1\ne G$ a finite $p$-group.
Then the following are equivalent:
\begin{enumerate}
\item $d_\fld(G)=1$;
\item $e_\fld(G)=1$;
\item $G\cong\mathbb{F}_{p^n}^+\le\fld$;
\item $G$ is isomorphic to a subgroup of ${\rm GL}_2(\fld)$.
\end{enumerate}
\end{prp}
\begin{proof}
``(3)$\iff$ (4)" is clear, since the finite $p$-groups of ${\rm GL}_2(\fld)$
are isomorphic to subgroups of the additive group ${\bf G}_a=(\fld,+)$.\\
``(1)$\Rightarrow$ (2)" is clear, because $e_\fld(G)=0\iff G=1\iff d_\fld(G)=0.$\\
``(2)$\Rightarrow$ (4)":\ If $G$ is not isomorphic to a subgroup of ${\rm PGL}_2(\fld)$ we
take $A:=D_\fld$ and $m=2$ in Proposition \ref{ess_dim_lower bound}.
By L\"uroth's theorem $\Aut_\fld(L)\cong {\rm PGL}_2(\fld)$ for any
intermediate field $\fld< L\le\fld(x_1)$, hence $2\le e_\fld(G)$.\\
 ``(3)$\Rightarrow$ (1)":\ This follows from Theorem \ref{el_abel_ext_fld_d_is_1}.
\end{proof}

\begin{cor}\label{d_k_is two}
Let $\fld$ be a field of characteristic $p>0$ and $G$ a finite $p$-group,
then $d_\fld(G)=2$ if and only if $e_\fld(G)=2$.
\end{cor}

\begin{cor}\label{el_ab_cycl_p_square}
Let $G\in\{C_p\times C_p, C_{p^2}\}$; assume moreover that $\fld$ is the prime field $\mathbb{F}_p$
in the case $G=C_p\times C_p$. Then $d_\fld(C_{p^2})=e_\fld(C_{p^2})=2=d_{\mathbb{F}_p}(C_p\times C_p)=e_{\mathbb{F}_p}(C_p\times C_p).$
\end{cor}
\begin{proof}
Note that $G$ is not isomorphic to a subgroup of ${\rm PGL}_2(\fld)$, hence
$2\le e_\fld(G)\le d_\fld(G)\le 2$, by Proposition \ref{ess_dim_squeeze}.
\end{proof}

\begin{prp}\label{e_k_el_ab}
Let $G$ be elementary-abelian of order $p^n$ with $n\ge 3$ and $\fld$ be any field of
characteristic $p$. Then $e_\fld(G)\le 2$.
\end{prp}
\begin{proof}
We can assume $\fld=\mathbb{F}_p$. We use the description of $G$ and further notation from
Theorem \ref{el_abel_ext_fld_d_is_1} and the arguments immediately before that. In particular
$U=\fld[Y_1,\cdots,Y_n]$ with $G$-action given by $(Y_j)g_i=Y_j-\delta_{ij}$.
Let $b\in U^G\backslash \fld$, so that $b$ is transcendental over $\fld$
and put $A=\fld[b, b Y_1 + b^2 Y_2 + b^3 Y_3 + ... + b^n Y_n]  <  U$.
Clearly $A$ is a polynomial subalgebra of $U$ with only two generators
on which $G$ acts faithfully. Note that for $Z:=\sum_{i=1}^n b^iY_i$ and
$\underline b:=(b,b^2,\cdots,b^n)$, we have $\fld(b)\otimes_\fld A\cong \fld(b)[Z]_{\bar b}$,
in the notation of Theorem \ref{el_abel_ext_fld_d_is_1}.
It follows that ${\rm Quot}(A)<{\rm Quot}(U)$, contains a trace-surjective algebra which is a
quotient of $U$ of Krull dimension $2$.  In particular  $e_{F_p}(G)\le 2$.
\end{proof}

\begin{rem}\label{ledet_rem}
\begin{enumerate}
\item The result in Proposition \ref{ess_dim_1} has been obtained in \cite{ledet07}
for arbitrary finite groups and infinite fields $\fld$, together with the
consequence that $e_{\fld}(C_p\times C_p)=1$.
In \cite{Kang_essdim_1} the groups with essential dimension one were classified
for all fields $k$.
\item The results in Corollaries \ref{d_k_is two}, \ref{el_ab_cycl_p_square}, Proposition \ref{e_k_el_ab}
and Theorem \ref{el_abel_ext_fld_d_is_1} show that the group invariants
$d_{\fld}(G)$ and $e_{\fld}(G)$ depend crucially on the choice of the ground field $\fld$.
\item Set $A:=\fld[x,y]$, the polynomial ring in two variables,
and $G:=\langle g\rangle\cong C_{p^2}$, then
there is a $G$-action on $A$ defined by $(x)g=x+y^{p-1}$ and $(y)g=y-1$.
Using \cite{nonlin} Lemma 5.2. one can show that $A\in\ts$ is
standard universal (i.e. a retract of $D_\fld$). Then
it follows from \ref{el_ab_cycl_p_square}, that $A$ is a basic object in $\ts$,
$Q(A)$ is an essential $G$-field and all essential $G$-fields are algebraic extensions of $Q(A)$.
\end{enumerate}
\end{rem}

\section{Concluding Remarks}\label{Concluding Remarks}

We conclude this article with a couple of open questions:
\begin{itemize}
\item In all cases where we know $d_\fld(G)$, there exists a basic algebra which
is an erasable polynomial ring. Is there always an erasable basic algebra?
Its ring of $G$-invariants would be stably polynomial.
\item If $\fld=\mathbb{F}_p$ we dare to conjecture that $d_\fld(G)=n=\mathrm{log}_p(|G|)$. This would follows from a positive answer
    to the first questions and Proposition 5.5 in \cite{nonlin}.
\item Does every erasable algebra in $\ts$ have a polynomial invariant ring?
\end{itemize}

\bibliography{bibl}

\begin{thebibliography}{10}

\bibitem{AG}
M.~Auslander and O.~Goldmann.
\newblock The {B}rauer group of a commutative ring.
\newblock {\em Trans.\ Amer.\ Math.\ Soc.}, 97:367--409, 1960.

\bibitem{berhuy_favi}
G~Berhuy and G~Favi.
\newblock Essential {D}imension: A functorial point of view (after {A}
  {M}erkurjev).
\newblock {\em Documenta Mathematica}, 8:270--330, 2003.

\bibitem{bou}
Nicolas Bourbaki.
\newblock {\em Groupes et alg\`{e}bres de {L}ie, Chap.\ {IV},{V},{VI}}.
\newblock Herman, Paris, 1968.

\bibitem{BH}
W.~Bruns and J.~Herzog.
\newblock {\em Cohen - Macaulay Rings}.
\newblock Cambridge University Press, 1993.

\bibitem{Buhler:Reichstein}
Joe Buhler and Zinovy Reichstein.
\newblock On the essential dimension of a finite group.
\newblock {\em Compos. Math.}, 106:159--179, 1997.

\bibitem{Kang_essdim}
Ming chang Kang.
\newblock Essential dimensions of finite groups.
\newblock {\em arXiv.math}, 0611673v2:1 -- 24, 2006.

\bibitem{chr}
S~U Chase, D~K Harrison, and A~Rosenberg.
\newblock Galois theory and galois cohomology of commutative rings.
\newblock {\em Memoirs of the AMS}, 52:15--33, 1965.

\bibitem{Kang_essdim_1}
H.~Chu, S.~J. Hu, M.~C. Kang, and J.~Zhang.
\newblock Groups with essential dimension one.
\newblock {\em Asian. Journal of Mathematics}, 12((2)):177--191, 2008.

\bibitem{B}
D.J.Benson.
\newblock {\em Polynomial Invariants of Finite Groups}.
\newblock Cambridge University Press, 1993.

\bibitem{costa}
D.L.Costa.
\newblock Retracts of polynomial rings.
\newblock {\em J. of Algebra}, 44:492--502, 1975.

\bibitem{eakin}
P.~Eakin.
\newblock A note on finite dimensional subrings of polynomial rings.
\newblock {\em Proc.\ Amer.\ Math.\ Soc.}, 31(1):75--80, 1972.

\bibitem{nonlin}
P.~Fleischmann and C.F. Woodcock.
\newblock Non-linear group actions with polynomial invariant rings and a
  structure theorem for modular {G}alois extensions.
\newblock {\em Proc. LMS}, 103(5):826--846, November 2011.

\bibitem{universal}
P.~Fleischmann and C.F. Woodcock.
\newblock Universal {G}alois {A}lgebras and {C}ohomology of $p$-groups.
\newblock {\em Journal of Pure and Applied Algebra}, 217(3):530--545, 2013.

\bibitem{gupta:2013}
N~Gupta.
\newblock On the cancellation problem for the affine space $\mathbb{A}^3$ in
  characteristic $p$.
\newblock {\em arXiv:1208.0483v2}, pages 1--9, 2013.

\bibitem{Jensen_Gen_pols}
C~U Jensen, A~Ledet, and N~Yui.
\newblock {\em Generic Polynomials}.
\newblock Cambridge University Press, 2002.

\bibitem{ledet07}
A~Ledet.
\newblock Finite groups of essential dimension one.
\newblock {\em J. of Algebra}, 311:31--37, 2007.

\bibitem{serre_how_to}
Jean-Pierre Serre.
\newblock How to use finite fields for problems concerning infinite fields.
\newblock {\em arXiv:0903.0517v2}, pages 1--12, 2009.

\bibitem{LSb}
L.~Smith.
\newblock {\em Polynomial Invariants of Finite Groups}.
\newblock A K Peters, 1995.

\bibitem{shpilrain}
V.Shpilrain and Jie-Tai Yu.
\newblock Polynomial retracts and the {J}acobian conjecture.
\newblock {\em Trans. of AMS}, 352(1):477--484, 1999.

\end{thebibliography}
\bibliographystyle{plain}

\address{P. Fleischmann and C.F. Woodcock, School of Mathematics, Statistics and Actuarial Science,
University of Kent, Canterbury, CT2 7NF, UK}
\email{P.Fleischmann@kent.ac.uk}
\email{C.F.Wooodcock@kent.ac.uk}
\end{document}